\documentclass[12pt,a4paper]{article}
\usepackage{amssymb}
\usepackage{amsmath, amscd}
\usepackage{amsthm}
\usepackage{mathtools}
\usepackage{mathrsfs}
\usepackage{booktabs}
\usepackage{perpage}
\usepackage[all,cmtip]{xy}
\usepackage{setspace}
\usepackage{amsbsy}
\usepackage[T1]{fontenc}
\usepackage{verbatim}
\usepackage{mathdots}
\usepackage{mathrsfs}
\usepackage{ifthen}
\usepackage{blkarray}
\usepackage{multirow}
\usepackage{enumerate}
\usepackage[dvipsnames]{xcolor}
\usepackage{tikz}
\usepackage{fullpage}
\usepackage{colonequals}
\usepackage{braket}
\usepackage{lipsum}
\usepackage{float}
\usepackage{bm}

\usetikzlibrary{matrix}
\usetikzlibrary{calc}

\usepackage{graphicx}

\makeatletter
\newcommand{\sigmaop}[1]{\mathop{\mathpalette\@sigmaop{#1}}\slimits@}
\newcommand{\@sigmaop}[2]{%
  \vphantom{\sum}%
  \sbox\z@{$\m@th#1\sum$}%
  \dimen@=\ht\z@ \advance\dimen@\dp\z@
  \dimen\tw@=\wd\z@
  \ifx#1\displaystyle\dimen@=.9\dimen@\fi
  \ooalign{%
    \hidewidth
    $\vcenter{\hbox{$\m@th#1#2$\kern.3\dimen\tw@}%
     \ifx#1\scriptstyle\kern-.25ex\fi}$\hidewidth\cr
    $\vcenter{\hbox{%
      \resizebox{!}{\dimen@}{$\m@th\boxtimes$}%
    }\ifx#1\scriptstyle\kern-.25ex\fi}$\cr
  }%
}
\makeatother

\mathtoolsset{showonlyrefs=true}

\numberwithin{equation}{subsection}

\newtheorem{theorem}{Theorem}[subsection]
\newtheorem{lemma}[theorem]{Lemma}

\newtheorem{conjecture}[theorem]{Conjecture}

\newtheorem{corollary}[theorem]{Corollary}
\newtheorem{definition}[theorem]{Definition}

\newtheorem{proposition}[theorem]{Proposition}

\newtheorem{assumption}[theorem]{Assumption}

\newtheorem*{thm1}{Theorem 1}
\newtheorem*{thm2}{Theorem 2}
\newtheorem*{thm3}{Theorem 3}

\newtheorem*{conj1}{Conjecture 1}
\newtheorem*{conj2}{Conjecture 2}

\theoremstyle{remark}

\newtheorem{rmk}[theorem]{Remark}

\newcommand{\GZip}{\mathop{\text{$G$-{\tt Zip}}}\nolimits}

\newcommand{\GpZip}{\mathop{\text{$G_p$-{\tt Zip}}}\nolimits}
\newcommand{\GF}{\mathop{\text{$G$-{\tt ZipFlag}}}\nolimits}

\newcommand{\VB}{\mathfrak{VB}}

\newskip\procskipamount
\newskip\interskipamount
\newskip\refskipamount
\newcommand{\procskip}{\vskip\procskipamount}
\newcommand{\interskip}{\vskip\interskipamount}
\newcommand{\refskip}{\vskip\refskipamount}

\newcommand{\procbreak}{\par
   \ifdim\lastskip<\procskipamount\removelastskip
   \penalty-100
   \procskip\fi
   \noindent\ignorespaces}

\newcommand{\titlebreak}{\par%
\ifdim\lastskip<\interskipamount\removelastskip%
\penalty10000%
\interskip\fi%
\noindent}%

\newcommand{\interbreak}{\par%
\ifdim\lastskip<\interskipamount\removelastskip%
\penalty-100%
\interskip\fi%
\noindent\ignorespaces}%

\newcommand{\refbreak}{\par%
\ifdim\lastskip<\refskipamount\removelastskip%
\penalty-100%
\refskip\fi%
\noindent\ignorespaces}%

\newcounter{listcounter}
\newcounter{deflistcounter}
\newcounter{equivcounter}

\newskip{\itemsepamount}
\newskip{\topsepamount}
\newenvironment{assertionlist}{%
  \begin{list}
    {\upshape (\arabic{listcounter})}
    {\setlength{\leftmargin}{18pt}
     \setlength{\rightmargin}{0pt}
     \setlength{\itemindent}{0pt}
     \setlength{\labelsep}{5pt}
     \setlength{\labelwidth}{13pt}
     \setlength{\listparindent}{\parindent}
     \setlength{\parsep}{0pt}
     \setlength{\itemsep}{\itemsepamount}
     \setlength{\topsep}{\topsepamount}
     \usecounter{listcounter}}}
  {\end{list}}

\newenvironment{definitionlist}{%
  \begin{list}
    {\upshape (\alph{deflistcounter})}
    {\setlength{\leftmargin}{18pt}
     \setlength{\rightmargin}{0pt}
     \setlength{\itemindent}{0pt}
     \setlength{\labelsep}{5pt}
     \setlength{\labelwidth}{13pt}
     \setlength{\listparindent}{\parindent}
     \setlength{\parsep}{0pt}
     \setlength{\itemsep}{\itemsepamount}
     \setlength{\topsep}{\topsepamount}
     \usecounter{deflistcounter}}}
  {\end{list}}

\newenvironment{equivlist}{%
  \begin{list}
    {\upshape (\roman{equivcounter})}
    {\setlength{\leftmargin}{18pt}
     \setlength{\rightmargin}{0pt}
     \setlength{\itemindent}{0pt}
     \setlength{\labelsep}{5pt}
     \setlength{\labelwidth}{13pt}
     \setlength{\listparindent}{\parindent}
     \setlength{\parsep}{0pt}
     \setlength{\itemsep}{\itemsepamount}
     \setlength{\topsep}{\topsepamount}
     \usecounter{equivcounter}}}
  {\end{list}}

% ...

\newcommand{\Acal}{{\mathcal A}}
\newcommand{\Bcal}{{\mathcal B}}

\newcommand{\Fcal}{{\mathcal F}}
\newcommand{\Gcal}{{\mathcal G}}

\newcommand{\Lcal}{{\mathcal L}}

\newcommand{\Ocal}{{\mathcal O}}
\newcommand{\Pcal}{{\mathcal P}}

\newcommand{\Tcal}{{\mathcal T}}
\newcommand{\Ucal}{{\mathcal U}}
\newcommand{\Vcal}{{\mathcal V}}

\newcommand{\Ycal}{{\mathcal Y}}
\newcommand{\Zcal}{{\mathcal Z}}

\newcommand{\gfr}{{\mathfrak g}}

\newcommand{\Sfr}{{\mathfrak S}}

\renewcommand{\AA}{\mathbb{A}}

\newcommand{\CC}{\mathbb{C}}

\newcommand{\EE}{\mathbb{E}}
\newcommand{\FF}{\mathbb{F}}
\newcommand{\GG}{\mathbb{G}}

\newcommand{\NN}{\mathbb{N}}

\newcommand{\QQ}{\mathbb{Q}}
\newcommand{\RR}{\mathbb{R}}

\newcommand{\ZZ}{\mathbb{Z}}

\newcommand{\Sscr}{{\mathscr S}}

\newcommand{\cent}{{\rm Cent}}

\DeclareMathOperator{\Gal}{Gal}

\DeclareMathOperator{\Span}{Span}

\DeclareMathOperator{\Lie}{Lie}

\DeclareMathOperator{\pr}{pr}

\DeclareMathOperator{\Rep}{Rep}

\DeclareMathOperator{\Sbt}{Sbt}
\DeclareMathOperator{\Sh}{Sh}
\DeclareMathOperator{\spec}{Spec}

\DeclareMathOperator{\Sch}{Sbt}

\DeclareMathOperator{\unip}{unip}

\DeclareMathOperator{\zip}{zip}
\DeclareMathOperator{\GS}{GS}

\DeclareMathOperator{\SL}{SL}
\DeclareMathOperator{\GL}{GL}

\DeclareMathOperator{\GSp}{GSp}

\DeclareMathOperator{\Sp}{Sp}

\DeclareMathOperator{\GU}{GU}

\newcommand{\shgx}{\Sh(\mathbf G, \mathbf X)}

\newcommand{\gx}{(\mathbf G, \mathbf X)}

\newcommand{\loccit}{{\em loc.\ cit. }}
\newcommand{\loccitn}{{\em loc.\ cit.}}

\newcommand{\diag}{{\rm diag}}

\renewcommand{\div}{{\rm div}}

\DeclareMathOperator{\flag}{flag}

\DeclareMathOperator{\Ind}{Ind}

\DeclareMathOperator{\Hasse}{Hasse}
\DeclareMathOperator{\Res}{Res}
\DeclareMathOperator{\Flag}{Flag}

\DeclareMathOperator{\ha}{ha}

\DeclareMathOperator{\Min}{Min}

\DeclareMathOperator{\Ha}{Ha}

\newcommand{\relmiddle}[1]{\mathrel{}\middle#1\mathrel{}}

\begin{document}

\author{Wushi Goldring and Jean-Stefan Koskivirta}

\title{Griffiths-Schmid conditions for automorphic forms via characteristic $p$} 

\date{}

\maketitle

\begin{abstract}
We establish vanishing results for spaces of automorphic forms in characteristic $0$ and characteristic $p$. We prove that for Hodge-type Shimura varieties, the weight of any nonzero automorphic form in characteristic $0$ satisfies the Griffiths--Schmid conditions, by purely algebraic, characteristic $p$ methods. We state a conjecture for general Hodge-type Shimura varieties regarding the vanishing of the space of automorphic forms in characteristic $p$ in terms of the weight. We verify this conjecture for unitary PEL Shimura varieties of signature $(n-1,1)$ at a split prime.
\end{abstract}

\section*{Introduction}

In this paper, we establish vanishing results for spaces of automorphic forms in both characteristic $0$ and characteristic $p$. Let $\gx$ be a Shimura datum, where $\mathbf{G}$ is a connected reductive $\QQ$-group. For a compact open subset $K\subset \mathbf{G}(\AA_f)$, we have a Shimura variety $\shgx_K$ defined over a number field $\mathbf{E}$. Let $\mathbf{P}\subset \mathbf{G}_{\mathbf{E}}$ be the parabolic subgroup attached to the Shimura datum (see \ref{subsec-automVB}). Choose a Borel subgroup $\mathbf{B}$ and a maximal torus $\mathbf{T}$ such that $\mathbf{T}\subset \mathbf{B}\subset \mathbf{P}$. Then, any algebraic $\mathbf{P}$-representation $(V,\rho)$ naturally gives rise to a vector bundle $\Vcal(\rho)$ on $\shgx_K$. Let $\mathbf{L}\subset \mathbf{P}$ denote the unique Levi subgroup of $\mathbf{P}$ containing $\mathbf{T}$.
Write $\Phi$ for the $\mathbf{T}$-roots of $\mathbf{G}$, and $\Phi_+, \Delta$ respectively for the positive roots and the simple roots with respect to $\mathbf{B}$. Let $I\subset \Delta$ be the subset of simple roots contained in $\mathbf{L}$. For $\lambda\in X^*(\mathbf{T})$, we consider the $\mathbf{P}$-representation $V_I(\lambda)=\Ind_{\mathbf{B}}^{\mathbf{P}}(\lambda)$ and denote by $\Vcal_I(\lambda)$ the associated vector bundle on $\shgx_K$. We call $\Vcal_I(\lambda)$ the automorphic vector bundle attached to the weight $\lambda$. The global sections of $\Vcal_I(\lambda)$ over $\shgx_K$ will be called automorphic forms of weight $\lambda$ and level $K$. 

We now restrict to the case when $\gx$ is of Hodge-type and $K$ is of the form $K=K_p K^p$ with $K_p\subset \mathbf{G}(\QQ_p)$ hyperspecial and $K^p\subset \mathbf{G}(\AA_f^p)$ is compact open (we say that $p$ is a prime of good reduction). Let $v|p$ be a place of $\mathbf{E}$ and write $\mathbf{E}_v$ for the completion of $\mathbf{E}$ at $v$. By results of Kisin (\cite{Kisin-Hodge-Type-Shimura}) and Vasiu (\cite{Vasiu-Preabelian-integral-canonical-models}), the variety $\shgx_K$ admits a smooth canonical model $\Sscr_K$ over $\Ocal_{\mathbf{E}_v}$. The vector bundle $\Vcal_I(\lambda)$ over $\shgx_K$ extends naturally to $\Sscr_K$. For any $\Ocal_{\mathbf{E}_v}$-algebra $F$ that is a field, we investigate which weights $\lambda$ admit nonzero automorphic forms with coefficients in $F$. In other words, we study the following set:
\begin{equation*}
C_K(F)\colonequals \{ \lambda\in X^*(\mathbf{T}) \mid H^0(\Sscr_K\otimes_{\Ocal_{\mathbf{E}_v}} F,\Vcal_I(\lambda)) \neq 0 \}.
\end{equation*}
It is a subcone (i.e additive submonoid) of $X^*(\mathbf{T})$. It is contained in the set $X^*_{+,I}(\mathbf{T})$ of $\mathbf{L}$-dominant characters, because $\Vcal_I(\lambda)=0$ for non $\mathbf{L}$-dominant $\lambda$ (by $\mathbf{L}$-dominant, we mean that it satisfies $\langle \lambda,\alpha^\vee \rangle \geq 0$ for all $\alpha\in I$). It suffices to consider the cases $F=\CC$ and $F=\overline{\FF}_p$. Indeed, note that if $F\subset F'$ then by flat base change, we have \[H^0(\Sscr_K\otimes_{\Ocal_{\mathbf{E}_v}} F',\Vcal_I(\lambda)) = H^0(\Sscr_K\otimes_{\Ocal_{\mathbf{E}_v}} F,\Vcal_I(\lambda))\otimes_F F',\]
therefore $C_K(F)=C_K(F')$. By \cite{Madapusi-Hodge-Tor}, there exists a smooth, toroidal compactification $\Sscr_K^{\Sigma}$ of $\Sscr_K$, where $\Sigma$ is a sufficiently fine cone decomposition. The vector bundle $\Vcal_I(\lambda)$ over $\shgx_K$ extends naturally to the toroidal compactification $\Sscr_K^{\Sigma}$. By results of Lan--Stroh in \cite{Lan-Stroh-stratifications-compactifications}, the Koecher principle holds, i.e there is an identification $H^0(\Sscr_K\otimes_{R} R,\Vcal_I(\lambda))=H^0(\Sscr^{\Sigma}_K\otimes_{R} R,\Vcal_I(\lambda))$ for all $\Ocal_{\mathbf{E}_v}$-algebra $R$ and all $\lambda\in X^*(\mathbf{T})$, except when $\dim(\shgx_K)=1$ and $\Sscr^{\Sigma}_K\setminus \Sscr_K\neq \emptyset$. We assume henceforth that $\dim(\shgx)>1$ or that $\Sscr_K$ is proper, so that the Koecher principle holds.

In general, the set $C_K(F)$ highly depends on the choice of the level $K$ (even in the case of the modular curve). For a subcone $C\subset X^*(\mathbf{T})$, define its saturated cone $\langle C\rangle$ as the set of $\lambda\in X^*(\mathbf{T})$ such that some positive multiple of $\lambda$ lies in $C$. The saturated cone $\langle C_K(F) \rangle$ is then independent of the level $K$ (\cite[Corollary 1.5.3]{Koskivirta-automforms-GZip}). Hence, it should be possible to give an expression for the saturated cone in terms of the root data of $\mathbf{G}$. Indeed, it is known (at least for $F=\CC$) that the cohomology of the Shimura variety $\Sscr_K\otimes_{\Ocal_{\mathbf{E}_v}} F$ can be expressed in terms of automorphic representations, and the theory of automorphic representations is to a large extent controlled by the root datum of the reductive group $\mathbf{G}$.

We first consider the case $F=\CC$. Griffiths--Schmid considered in \cite{Griffiths-Schmid-homogeneous-complex-manifolds} the following set of characters:
\begin{equation}
C_{\GS}=\left\{ \lambda\in X^{*}(\mathbf{T}) \ \relmiddle| \ 
\parbox{6cm}{
$\langle \lambda, \alpha^\vee \rangle \geq 0 \ \textrm{ for }\alpha\in I, \\
\langle \lambda, \alpha^\vee \rangle \leq 0 \ \textrm{ for }\alpha\in \Phi_+ \setminus \Phi_{\mathbf{L},+}$}
\right\}.
\end{equation}
Here $\Phi_{\mathbf{L},+}$ denotes the positive $\mathbf{T}$-roots in $\mathbf{L}$. We call this cone the Griffiths-Schmid cone. The following seems to be known to experts, but as far as we know there is no reference where this result is explicitly stated.

\begin{thm1}
Let $\gx$ be any Hodge-type Shimura datum. Let $\lambda\in X^*(\mathbf{T})$ be a character and assume that $\lambda\notin C_{\GS}$. Then we have $H^0(\shgx_K,\Vcal_I(\lambda))=0$.
\end{thm1}

In other words, this theorem amounts to the inclusion $C_K(\CC)\subset C_{\GS}$. We note that the equality $\langle C_K(\CC)\rangle =C_{\GS}$ is expected in general. It seems possible to show the above theorem using the theory of Lie algebra cohomology. In this paper, we give a proof based on purely characteristic $p$ methods, which is a novel aspect of our approach. For an automorphic form $f$ in characteristic zero, we may consider the reduction of $f$ modulo $v$ for all except finitely many places $v$ of $\mathbf{E}$. Then, our approach is to use the geometric structure (namely the Ekedahl--Oort stratification) of the special fiber at $v$ to extract information about the weight of $f$. 

In our proof of Theorem 1, only weak information at each prime is sufficient to obtain the result because we are able to reduce $f$ at infinitely many places. On the other hand, a more difficult question is to fix a prime $p$ (of good reduction) and study the cone $C_K(\overline{\FF}_p)$. Similarly to the characteristic zero case, the saturated cone $\langle C_K(\overline{\FF}_p) \rangle$ is independent of $K$ and we expect that it can be expressed in terms of root data. However, it also depends in general on the prime $p$. We have conjectured the following (\cite[Conjecture C]{Goldring-Koskivirta-global-sections-compositio}):
\begin{conj1}
We have $\langle C_K(\overline{\FF}_p) \rangle = \langle C_{\zip} \rangle$.
\end{conj1}
Here, the cone $C_{\zip}$ is an entirely group-theoretical object defined using the stack of $G$-zips defined by Moonen--Wedhorn (\cite{Moonen-Wedhorn-Discrete-Invariants}) and
Pink--Wedhorn--Ziegler (\cite{Pink-Wedhorn-Ziegler-zip-data, PinkWedhornZiegler-F-Zips-additional-structure}). Specifically, write $S_K\colonequals \Sscr_K\otimes_{\Ocal_{\mathbf{E}_v}}\overline{\FF}_p$. Since $K_p$ is hyperspecial, $\mathbf{G}$ admits a $\ZZ_p$-reductive model $\Gcal$. Set $G\colonequals \Gcal\otimes_{\ZZ_p} \FF_p$, and write similarly $T,L$ for the reduction of $\mathbf{T},\mathbf{L}$ respectively. By results of Zhang (\cite{Zhang-EO-Hodge}) there exists a smooth map $\zeta\colon S_K\to \GZip^\mu$ where $\GZip^\mu$ is the stack of $G$-zips of type $\mu$ (here $\mu$ is a cocharacter of $G_{\overline{\FF}_p}$ whose centralizer is $L$). The map $\zeta$ is also surjective by \cite[Corollary 3.5.3(1)]{Shen-Yu-Zhang-EKOR}. The automorphic vector bundles $\Vcal_I(\lambda)$ also exist on the stack $\GZip^\mu$ (see \cite[\S 2.4]{Imai-Koskivirta-vector-bundles}), compatibly with the map $\zeta$. We defined $C_{\zip}$ (\cite[(1.2.3)]{Koskivirta-automforms-GZip}) as the set of $\lambda\in X^*(\mathbf{T})$ such that $H^0(\GZip^\mu,\Vcal_I(\lambda))\neq 0$. The space $H^0(\GZip^\mu,\Vcal_I(\lambda))$ can be interpreted in terms of representation theory of reductive groups (\cite[Theorem 3.7.2]{Koskivirta-automforms-GZip}, \cite[Theorem 1]{Imai-Koskivirta-vector-bundles}).

Conjecture 1 was proved in \cite[Theorem D]{Goldring-Koskivirta-global-sections-compositio} for Hilbert--Blumenthal Shimura varieties, Siegel threefolds and Picard surfaces (at split primes). The Hilbert--Blumenthal case was also treated independently by Diamond--Kassaei in \cite[Corollary 1.3]{Diamond-Kassaei} using different methods and a different formulation. In the preprint \cite{Goldring-Koskivirta-divisibility}, it is proved in the cases $G=\GSp(6)$, $\GU(r,s)$ for $r+s\leq 4$ (except when $r=s=2$ and $p$ is inert). The set $C_{\zip}$ is much more tractable than $\langle C_K(\overline{\FF}_p) \rangle$, but is still difficult to determine in general. We can use Conjecture 1 in order to gain intuition about the cone $\langle C_K(\overline{\FF}_p) \rangle$. Conversely, facts pertaining to automorphic forms and their weights should have an equivalent group-theoretical statement on the level of the stack $\GZip^\mu$. For example, using reduction modulo $p$, one shows easily that $C_K(\CC)\subset C_K(\overline{\FF}_p)$ (see \cite[Proposition 1.8.3]{Koskivirta-automforms-GZip}), hence also $\langle C_K(\CC) \rangle\subset \langle C_K(\overline{\FF}_p)\rangle$. Since it is expected that $\langle C_K(\CC) \rangle=C_{\GS}$ in general, one should expect an inclusion $C_{\GS}\subset \langle C_{\zip}\rangle$. This fact is highly nontrivial, and was indeed proved in general in the recent preprint \cite[Theorem Theorem 6.4.2]{Goldring-Imai-Koskivirta-weights}, as a sanity check for Conjecture 1 to hold.

In this paper, our second goal is to seek an upper bound approximation of $\langle C_K(\overline{\FF}_p)\rangle$. To gain intuition, we first consider the cone $C_{\zip}$ and determine an upper bound for it. We define in section \ref{sec-unip-cone} the unipotent-invariance cone $C_{\unip}\subset X^{*}(\mathbf{T})$ and show that $C_{\zip} \subset C_{\unip}$. When $G$ is split over $\FF_p$ or $\FF_{p^2}$ and $P$ is defined over $\FF_p$, we can give concrete equations for an upper bound of $C_{\zip}$. Let $W_L=W(L,T)$ be the Weyl group of $L$. Note that $W_L \rtimes \Gal(\FF_{p^2}/\FF_p)$ acts naturally on the set $\Phi_+\setminus \Phi_{L,+}$. Let $\Ocal\subset \Phi_+\setminus \Phi_{L,+}$ be an orbit under the action of $W_{L}\rtimes \Gal(\FF_{p^2}/\FF_p)$ and let $S\subset \Ocal$ be any subset. Set
\begin{equation}\label{ineq-subset-intro}
\Gamma_{\Ocal,S,p}(\lambda)\colonequals \sum_{\alpha \in \Ocal\setminus S} \langle \lambda,\alpha^\vee \rangle \ + \ \frac{1}{p} \sum_{\substack{\alpha\in S}} \langle  \lambda, \alpha^\vee \rangle
\end{equation}
Define $C_{\Ocal}$ as the set of $\lambda\in X^*(\mathbf{T})$ such that $\Gamma_{\Ocal,S}(\lambda)\leq 0$ for all subsets $S\subset \Ocal$. Then we have
\begin{equation*}
    C_{\zip}\subset \bigcap_{\substack{ \textrm{orbits} \\ \Ocal \subset \Phi_+\setminus \Phi_{L,+}}} C_{\Ocal}.
\end{equation*}
Only certain choices of $(\Ocal,S)$ will contribute non-trivially to the above intersection, but for a general group it is unclear to us how to determine the important pairs $(\Ocal, S)$. By Conjecture 1, we can expect the following:
\begin{conj2}
Let $S_K$ be the special fiber of a Hodge-type Shimura variety at a prime $p$ of good reduction which splits in $\mathbf{E}$. Furthermore, assume that the attached reductive $\FF_p$-group $G$ is split over $\FF_{p^2}$. Then if $f\in H^0(S_K,\Vcal_I(\lambda))$ is a nonzero automorphic form of weight $\lambda\in X^*(\mathbf{T})$, we have $\Gamma_{\Ocal,S,p}(\lambda)\leq 0$ for all $W_{L}\rtimes \Gal(\FF_{p^2}/\FF_p)$-orbit $\Ocal\subset \Phi_+\setminus \Phi_{L,+}$ and all subsets $S\subset \Ocal$.
\end{conj2}
We now consider the case of Shimura varieties attached to a unitary similitude group $\mathbf{G}$ such that $\mathbf{G}_\RR\simeq \GU(n-1,1)$. We choose a split prime $p$ of good reduction. In this case $G\simeq \GL_{n-1,\FF_p}\times \GG_{\textrm{m},\FF_p}$. We parametrize weights by $n$-tuples $(k_1,\dots,k_n)\in \ZZ$. We prove Conjecture 2 in this case. More precisely, we have the following:
\begin{thm2}\label{th-CF-Min-intro}
Let $S_K$ be the good reduction special fiber of a unitary Shimura variety of signature $(n-1,1)$ at a split prime $p$. Let $f\in H^0(S_K,\Vcal_I(\lambda))$ be a nonzero mod $p$ automorphic form and write $\lambda=(k_1,\dots,k_n)\in \ZZ^n$. Then we have: 
\begin{equation}
    \sum_{i=1}^j (k_i-k_n) + \frac{1}{p} \sum_{i=j+1}^{n-1} (k_i-k_n) \leq 0 \quad \textrm{ for all }j=1,\dots, n-1.
\end{equation}
\end{thm2}
The inequalities appearing in the statement of the theorem are of the form $\Gamma_{\Ocal,S}(\lambda)\leq 0$, as in Conjecture 2. In this case, the set $\Phi_+\setminus \Phi_{L,+}$ consists of a single orbit under the group $W_L$. Furthermore, we only need consider the sets $S\subset \Phi_+\setminus \Phi_{L,+}$ which satisfy the property that if $w\in S$, then any $w'\geq w$ is also in $S$, because one sees easily that the other sets do not contribute. This gives the $n$ inequalities in Theorem 2. It is compatible with Theorem 1 in the following sense. In our convention of positivity, we have
\begin{equation*}
    C_{\GS}=\{\lambda=(k_1,\dots,k_n)\in \ZZ^n \ \mid \ k_n\geq k_1\geq \dots \geq k_{n-1} \}.
\end{equation*}
Note that $C_{\GS}$ is the set of $L$-dominant characters $\lambda\in X^*_{+,L}(T)$ satisfying the condition $k_1\leq k_n$. If we let $p$ go to infinity in the inequality corresponding to $j=1$ in Theorem 2, we deduce that the weight $\lambda=(k_1,\dots,k_n)$ of any characteristic zero automorphic form satisfies $k_1\leq k_n$, hence lies in $C_{\GS}$.

We briefly explain the proof of Theorem 2. First, we consider the flag space of $S_K$, which is a $P/B$-fibration $\pi_K\colon \Flag(S_K)\to S_K$. It carries a family of line bundles $\Vcal_{\flag}(\lambda)$ for $\lambda\in X^*(\mathbf{T})$ such that $\pi_{K,*}(\Vcal_{\flag}(\lambda))=\Vcal_{I}(\lambda)$. Furthermore, it carries a stratification $(\Flag(S_K)_w)_{w\in W}$ defined as the fibers of a natural map $\psi_K\colon \Flag(S_K)\to \left[B\backslash G / B\right]$. For each $w\in W$, we define a cone $C_{K,w}\subset X^*(\mathbf{T})$ as the set of $\lambda$ such that the line bundle $\Vcal_{\flag}(\lambda)$ admits nonzero sections on the Zariski closure $\overline{\Flag(S_K)}_w$. There is a natural subcone $C_{\Hasse,w}\subset C_{K,w}$ given by the weights of sections which arise by pullback from the stack $\left[B\backslash G / B\right]$ via $\psi_K$. We say that the stratum $\Flag(S_K)_w$ is Hasse-regular if $\langle C_{\Hasse,w}\rangle =\langle C_{K,w}\rangle$. Let $w_0$ and $w_{0,L}$ be the longest elements in $W$ and $W_L$ respectively. Set $z=w_{0,L}w_0$. The projection $\pi_K$ restricts to a map $\pi_K\colon \overline{\Flag(S_K)}_z\to S_K$ which is finite etale on the open subset $\Flag(S_K)_z$. The proof of Theorem 2 uses the following result as a starting point:
\begin{thm3}
Let $S_K$ be the good reduction special fiber of a unitary Shimura variety of signature $(n-1,1)$ at a split prime. For any $w\in W$ such that $w\leq z$, the stratum $\Flag(S_K)_w$ is Hasse-regular.
\end{thm3}
We conjecture that the above also generalizes for all Hodge-type Shimura varieties when $G$ is split over $\FF_p$. Concretely, this theorem implies the following: Let $f$ be any nonzero section of $\Vcal_{\flag}(\lambda)$ on $\overline{\Flag(S_K)}_z$ for $\lambda=(k_1,\dots, k_n)\in \ZZ^n$. Then $\lambda$ satisfies $k_i-k_n\leq 0$ for all $i=1,\dots, n-1$. In particular, let $f$ be any nonzero automorphic form in characteristic $p$, of weight $\lambda$. We may view $f$ as a global section of the line bundle $\Vcal_{\flag}(\lambda)$ on $\Flag(S_K)$, using the relation $\pi_{K,*}(\Vcal_{\flag}(\lambda))=\Vcal_I(\lambda)$. If $\lambda\notin C_{\GS}$, then the restriction of $f$ to the stratum $\Flag(S_K)_z$ is zero. We expect this result to generalize to all Hodge-type cases at split primes of good reduction.

We prove Theorem 2 as a consequence of Theorem 3, by using a suitable sequence of elements $w_1,\dots, w_N$ in $W$ starting at $w_1=w_0$ and ending at $w_N=z$. For each $1\leq i\leq N-1$, $w_{i+1}$ is a lower neighbour of $w_i$ with respect to the Bruhat order on $W$. Furthermore, the flag stratum corresponding to $w_{i+1}$ is cut out inside the Zariski closure of $\Flag(S_K)_{w_i}$ by a certain partial Hasse invariant $\Ha_i$. It then follows easily that the weight of any nonzero global section of $\Vcal_{\flag}(\lambda)$ is the sum of the weights of $\Ha_i$ and of an element of $C_{\GS}$, which proves the result.

\vspace{0.8cm}

\noindent
{\bf Acknowledgements.}
This work was supported by JSPS KAKENHI Grant Number 21K13765. W.G. thanks the Knut \& Alice Wallenberg Foundation for its support under grants KAW 2018.0356 and Wallenberg Academy Fellow KAW 2019.0256, and the Swedish Research Council for its support under grant \"AR-NT-2020-04924.

\section{Weights of automorphic forms}

\subsection{Automorphic forms on Shimura varieties}

\subsubsection{Shimura varieties} \label{subsec-Shimura}

Let $\gx$ be a Shimura datum of Hodge-type \cite[2.1.1]{Deligne-Shimura-varieties}. In particular, $\mathbf{G}$ is a connected, reductive group over $\QQ$. Furthermore, $\mathbf{X}$ gives rise to a well-defined $\mathbf{G}(\overline{\QQ})$-conjugacy class of cocharacters $\{\mu\}$ of $\mathbf{G}_{\overline{\QQ}}$. Let $\mathbf{E}=\mathbf{E}(\mathbf{G},\mathbf{X})$ be the reflex field of $\gx$ (i.e. the field of definition of $\{\mu\}$) and $\Ocal_\mathbf{E}$ its ring of integers. If $K \subset \mathbf{G}(\AA_f)$ is an open compact subgroup, write $\shgx_{K}$ for Deligne's canonical model at level $K$ over $\mathbf{E}$ (see \cite{Deligne-Shimura-varieties}). When $K$ is small enough, $\shgx_K$ is a smooth, quasi-projective scheme over $\mathbf{E}$. Fix a finite set of "bad" primes $S$ and a compact open subgroup $K\subset \mathbf{G}(\AA_f)$ of the form
\begin{equation}
    K=K_S \times K^S 
\end{equation}
where $K_S\subset \mathbf{G}(\QQ_S)$ and $K^S=\mathbf{G}(\widehat{\ZZ}^S)$, where $\QQ_S=\prod_{p\in S} \QQ_p$ and $\widehat{\ZZ}^S=\prod_{p\notin S}\ZZ_p$. For all $p\notin S$, the Shimura variety $\shgx_K$ has good reduction at all primes above $p$. In particular, for each $p\notin S$, the group $\mathbf{G}_{\QQ_p}$ is unramified, so there exists a reductive $\ZZ_p$-model $\Gcal$, such that $G\colonequals \Gcal\otimes_{\ZZ_p}\mathbb{F}_p$ is connected. For any place $v$ above $p$ in $\mathbf{E}$, Kisin (\cite{Kisin-Hodge-Type-Shimura}) and Vasiu (\cite{Vasiu-Preabelian-integral-canonical-models}) constructed a smooth canonical model $\Sscr_K$ of $\shgx_K$ over $\Ocal_{\mathbf{E},v}$. By glueing, we obtain a smooth $\Ocal_{\mathbf{E}}\left[\frac{1}{N}\right]$-model, that we will abusively continue to denote by $\Sscr_K$, where $N\geq 1$ is an integer divisible by all the primes in $S$. We denote its mod $p$ reduction by $S_{K,p}\colonequals \Sscr_K\otimes_{\Ocal_{\mathbf{E},v}} \overline{\FF}_p$ (we simply write $S_K$ when the choice of $p$ is fixed). We will have to extend the ring of definition so that all objects we consider are defined over that ring. Therefore, we let $R$ be a ring of the form $\Ocal_{\mathbf{E}'}\left[\frac{1}{N'}\right]$ for a number field $\mathbf{E}\subset \mathbf{E}'$ and an integer $N'$ divisible by $N$. We will freely change $R$ to a suitable extention by modifying $\mathbf{E}'$ and $N'$.

\subsubsection{Automorphic vector bundles}\label{subsec-automVB}
A cocharacter $\mu \in \{\mu\}$ induces a decomposition of $\gfr\colonequals\Lie(\mathbf{G}_\CC)$ as $\gfr=\bigoplus_{n\in \ZZ} \gfr_n$, where $\gfr_n$ is the subspace where $\GG_{\textrm{m},\CC}$ acts on $\gfr$ by $x\mapsto x^n$ via $\mu$. It gives rise to an opposite pair of parabolic subgroups $\mathbf{P}_{\pm}(\mu)$ such that $\Lie(\mathbf{P}_+(\mu))$ (resp. $\Lie(\mathbf{P}_-(\mu))$ is the direct sum of $\gfr_n$ for $n\geq 0$ (resp. $n\leq 0$). We set $\mathbf{P}=\mathbf{P}_-(\mu)$. Let $(\mathbf{B},\mathbf{T})$ be a Borel pair of $\mathbf{G}_\CC$ such that $\mathbf{B}\subset \mathbf{P}$ and such that $\mu\colon \GG_{\textrm{m},\CC}\to \mathbf{G}_\CC$ factors through $\mathbf{T}$. As usual, $X^*(\mathbf{T})$ denotes the group of characters of $\mathbf{T}$. Let $\mathbf{B}^+$ be the opposite Borel subgroup (i.e the unique Borel subgroup such that $\mathbf{B}^+\cap \mathbf{B}=\mathbf{T}$). Let $\Phi\subset X^*(\mathbf{T})$ be the set of $\mathbf{T}$-roots of $G$ and $\Phi_+\subset \Phi$ the system of positive roots with respect to $\mathbf{B}^+$ (i.e. $\alpha \in \Phi_+$ whenever the $\alpha$-root group $U_{\alpha}$ is contained in $\mathbf{B}^+$). We use this convention to match those of the previous publications \cite{Goldring-Koskivirta-Strata-Hasse,Koskivirta-automforms-GZip}. Let $\Delta\subset \Phi_+$ be the set of simple roots. Let $I\subset \Delta$ denote the set of simple roots of the unique Levi subgroup $\mathbf{L}\subset \mathbf{P}$ containing $\mathbf{T}$ (note that $\mathbf{L}$ is the centralizer of $\mu$).

We may assume that there exists a reductive, smooth group scheme $\Gcal$ over $\ZZ[\frac{1}{N'}]$ such that $\Gcal\otimes_{\ZZ[1/N']} \QQ\simeq \mathbf{G}$ and that $\mu$ extends to a cocharacter of $\Gcal\otimes_{\ZZ[1/N']} R$. In particular, we obtain a parabolic subgroup $\Pcal\subset \Gcal\otimes_{\ZZ[1/N']} R$ that extends $\mathbf{P}$. The $R$-scheme $\Sscr_K$ carries a universal $\Pcal$-torsor afforded by the Hodge filtration. This torsor yields a natural functor
\begin{equation}\label{funct-V}
     \Vcal \colon \Rep_{R}(\Pcal) \longrightarrow \VB(\Sscr_K) 
\end{equation}
where $\Rep_{R}(\Pcal)$ denotes the category of algebraic $R$-representations of $\Pcal$, and $\VB(\Sscr_K)$ is the category of vector bundles on $\Sscr_K$. Furthermore, the functor $\Vcal$ commutes in an obvious sense with change of level. The vector bundles of the form $\Vcal(\rho)$ for $\rho\in \Rep_R(\Pcal)$ are called \emph{automorphic vector bundles} in \cite[III. Remark 2.3]{Milne-ann-arbor}.

Let $\lambda\in X^*(\mathbf{T})$ be an $\mathbf{L}$-dominant character, by which we mean that $\langle \lambda,\alpha^\vee \rangle\geq 0$ for all $\alpha\in I$. Set $\mathbf{V}_{I}(\lambda)=H^0(\mathbf{P}/\mathbf{B},\Lcal_\lambda)$, where $\Lcal_\lambda$ is the line bundle on $\mathbf{P}/\mathbf{B}$ attached to
$\lambda$. It is the unique irreducible representation of $\mathbf{P}$ over $\overline{\QQ}$ of highest weight $\lambda$. After possibly extending $R$, we may assume that $\mathbf{V}_{I}(\lambda)$ admits a natural model over $R$, namely $\mathbf{V}_{I}(\lambda)_{R} \colonequals H^0(\Pcal/\Bcal,\Lcal_\lambda)$, where $\Bcal$ is a $\ZZ[1/N']$-Borel subgroup of $\Gcal$ extending $\mathbf{B}$. We denote by $\Vcal_I(\lambda)$ the vector bundle on $\Sscr_K$ attached to the $\Pcal$-representation $\mathbf{V}_{I}(\lambda)_{R}$.

\subsubsection{The stack of $G$-zips}\label{subsec-GZip}
Let $p$ be a prime number and $q$ a $p$-power. Fix an algebraic closure $k$ of $\FF_q$. For a $k$-scheme $X$, we denote by $X^{(q)}$ its $q$-th power Frobenius twist and by $\varphi\colon X\to X^{(q)}$ its relative Frobenius. Let $\sigma\in \Gal(k/\FF_q)$ be the automorphism $x\mapsto x^q$. If $G$ is a connected, reductive group over $\FF_q$ and $\mu\colon \GG_{\mathrm{m},k}\to G_k$ is a cocharacter, we call the pair $(G,\mu)$ a cocharacter datum over $\FF_q$. In the context of Shimura varieties, we always take $q=p$, and $G$ will be the reduction modulo $p$ of $\mathbf{G}$ at a prime of good reduction. To the pair $(G,\mu)$, we can attach (functorially) a finite smooth stack $\GZip^\mu$ called the stack of $G$-zips of type $\mu$. It was introduced by Moonen--Wedhorn and Pink--Wedhorn--Ziegler in \cite{Moonen-Wedhorn-Discrete-Invariants,Pink-Wedhorn-Ziegler-zip-data,PinkWedhornZiegler-F-Zips-additional-structure}. As in section \ref{subsec-automVB}, $\mu$ gives rise to two opposite parabolic subgroups $P_{\pm}(\mu)\subset G_k$. We set $P\colonequals P_-(\mu)$ and $Q\colonequals P_+(\mu)^{(q)}$. Let $L\colonequals \cent(\mu)$ be the centralizer of $\mu$, it is a Levi subgroup of $P$. Put $M\colonequals L^{(q)}$, which is a Levi subgroup of $Q$. We have a Frobenius map $\varphi\colon L\to M$. The tuple $\Zcal\colonequals (G,P,Q,L,M,\varphi)$ is called the zip datum attached to $(G,\mu)$.

Let $\theta_L^P\colon P\to L$ be the projection onto the Levi subgroup $L$ modulo the unipotent radical $R_{\mathrm{u}}(P)$. Define $\theta_M^Q\colon Q\to M$ similarly. The zip group of $\Zcal$ is defined by
\begin{equation}\label{eq-Edef}
    E\colonequals\{ (x,y)\in P\times Q \mid \varphi(\theta^P_L(x))=\theta_M^Q(y) \}.
\end{equation}
Let $E$ act on $G$ by the rule $(x,y)\cdot g\colonequals xgy^{-1}$. The stack of $G$-zips $\GZip^\mu$ can be defined as the quotient stack
\begin{equation}
    \GZip^\mu \colonequals \left[E\backslash G_k\right].
\end{equation}
To any $E$-representation $(V,\rho)$, one attaches a vector bundle $\Vcal(\rho)$ on $\GZip^\mu$, as explained in \cite[\S 2.4.2]{Imai-Koskivirta-vector-bundles} using the associated sheaf construction (\cite[I.5.8]{jantzen-representations}). In particular, a $P$-representation $(V,\rho)$ gives rise to an $E$-representation via the first projection $\pr_1\colon E\to P$, thus to a vector bundle $\Vcal(\rho)$. Choose a Borel pair $(B,T)$ of $G_k$ such that $B\subset P$ and such that $\mu$ factors through $T$. For $\lambda\in X^*(T)$, define $V_I(\lambda)$ as the $P$-representation $\Ind_B^P(\lambda)=H^0(P/B,\Lcal_\lambda)$ similarly to section \ref{subsec-automVB}. The vector bundle on $\GZip^\mu$ attached to $V_I(\lambda)$ is denoted again by $\Vcal_I(\lambda)$.

We now explain the connection with Shimura varieties. We return to the setting of section \ref{subsec-Shimura}. Fix a prime $p\notin S$ of good reduction and let $G\colonequals \Gcal\otimes_{\ZZ[\frac{1}{N'}]} \FF_p$. Write again $\mu$ for the cocharacter of $G_{\overline{\FF}_p}$ obtained by reduction mod $p$. We obtain a cocharacter datum $(G,\mu)$ over $\FF_p$, and hence a zip datum $(G,P,L,Q,M,\varphi)$ and a stack of $G$-zips $\GZip^\mu$. Write $S_{K,p}\colonequals \Sscr\otimes_R \overline{\FF}_p$. Zhang (\cite[4.1]{Zhang-EO-Hodge}) constructed a smooth morphism
\begin{equation}\label{zeta-Shimura}
\zeta \colon S_{K,p}\to \GZip^\mu. 
\end{equation}
This map is also surjective by \cite[Corollary 3.5.3(1)]{Shen-Yu-Zhang-EKOR}. Furthermore, the automorphic vector bundle $\Vcal_I(\lambda)$ defined on $S_{K,p}$ using the functor \eqref{funct-V} coincides with the pullback via $\zeta$ of the vector bundle $\Vcal_I(\lambda)$ defined on $\GZip^\mu$.

\subsubsection{Toroidal compactification} \label{sec-tor}

By \cite[Theorem 1]{Madapusi-Hodge-Tor}, there is a sufficiently fine cone decomposition $\Sigma$ and a toroidal compactification $\Sscr_K^\Sigma$ of $\Sscr_K$ over $\Ocal_{\mathbf{E},v}$. Again, by glueing we may assume that there exists a toroidal compactification of $\Sscr_K$ over the ring $R$, that we denote again by $\Sscr_K^\Sigma$. Furthermore, the family $(\Vcal_I(\lambda))_{\lambda\in X^*(\mathbf{T})}$ admits a canonical extension $(\Vcal^{\Sigma}_I(\lambda))_{\lambda\in X^*(\mathbf{T})}$ to $\Sscr^{\Sigma}_K$. For a prime $p$, set $S_{K,p}^\Sigma\colonequals \Sscr_K^\Sigma\otimes_R \overline{\FF}_p$. By \cite[Theorem 6.2.1]{Goldring-Koskivirta-Strata-Hasse}, the map $\zeta\colon S_{K,p}\to \GZip^\mu$ extends naturally to a map
\begin{equation}
    \zeta^\Sigma\colon S_{K,p}^{\Sigma}\to \GZip^\mu.
\end{equation}
Furthermore, by \cite[Theorem 1.2]{Andreatta-modp-period-maps}, the map $\zeta^\Sigma$ is smooth. Since $\zeta$ is surjective, $\zeta^\Sigma$ is also surjective. Moreover, \cite[Proposition 6.20]{Wedhorn-Ziegler-tautological} shows that any connected component $S^\circ\subset S_{K,p}^{\Sigma}$ intersects the unique zero-dimensional stratum. Since the map $\zeta^\Sigma \colon S^\circ \to \GZip^\mu$ is smooth, its image is open, hence surjective. Therefore, the restriction of $\zeta^\Sigma\colon S_{K,p}^{\Sigma}\to \GZip^\mu$ to any connected component is also surjective.

By construction, the pullback of $\Vcal_I(\lambda)$ via $\zeta$ coincides with the canonical extension $\Vcal^{\Sigma}_I(\lambda)$. We have the following Koecher principle:

\begin{theorem}[{\cite[Theorem 2.5.11]{Lan-Stroh-stratifications-compactifications}}] \label{thm-koecher}
Let $F$ be a field which is an $R$-algebra. The natural map
\begin{equation}
    H^0(\Sscr^{\Sigma}_K\otimes_R F,\Vcal^{\Sigma}_I(\lambda)) \to H^0(\Sscr_K\otimes_R F,\Vcal_I(\lambda)) 
\end{equation}
is a bijection, except when $\dim(\Sscr_K)=1$ and $\Sscr^{\Sigma}_K \setminus \Sscr_K\neq \emptyset$.
\end{theorem}

We will only consider Shimura varieties satisfying the condition $\dim(\Sscr_K)>1$ or $\Sscr^{\Sigma}_K \setminus \Sscr_K\neq \emptyset$.

\subsection{Weight cones of automorphic forms}
\subsubsection{Griffiths--Schmid conditions}

The motivation of this paper is to study the possible weights of automorphic forms over various fields. Specifically, for any field $F$ which is an $R$-algebra, define
\begin{equation}
    C_K(F)\colonequals \{ \lambda\in X^*(\mathbf{T}) \mid H^0(\Sscr_K\otimes_{R} F,\Vcal_I(\lambda)) \neq 0 \}.
\end{equation}
By the Koecher principle (Theorem \ref{thm-koecher}), we may replace the pair $(\Sscr_K\otimes_R F,\Vcal_I(\lambda))$ with the pair $(\Sscr^{\Sigma}_K\otimes_R F,\Vcal^{\Sigma}_I(\lambda))$ in the definition of $C_K(F)$.

As explained in the introduction, there are two main cases to consider, namely $F=\CC$ and $F=\overline{\FF}_p$ for a prime number $p$ of good reduction. We first consider the case $F=\CC$. The space $H^0(\Sh_K(\mathbf{G},\mathbf{X}),\Vcal_I(\lambda))$ is the space of classical, characteristic zero automorphic forms of weight $\lambda$ and level $K$. Therefore, the set $C_K(\CC)$ is the set of possible weights of nonzero automorphic forms in characteristic $0$. It is a subcone of $X^*(\mathbf{T})$ (by "cone", we mean an additive monoid containing zero). Write $\Phi_{\mathbf{L},+}$ for the set of positive $\mathbf{T}$-roots of $\mathbf{L}$. The Griffiths--Schmid cone $C_{\GS}$ is defined as follows.
\begin{equation}
C_{\GS}=\left\{ \lambda\in X^{*}(\mathbf{T}) \ \relmiddle| \ 
\parbox{6cm}{
$\langle \lambda, \alpha^\vee \rangle \geq 0 \ \textrm{ for }\alpha\in I, \\
\langle \lambda, \alpha^\vee \rangle \leq 0 \ \textrm{ for }\alpha\in \Phi_+ \setminus \Phi_{\mathbf{L},+}$}
\right\}.
\end{equation}
The conditions defining the cone $C_{\GS}$ were first introduced by Griffiths--Schmid in \cite{Griffiths-Schmid-homogeneous-complex-manifolds}. It is expected that $C_{K}(\CC)\subset C_{\GS}$ for general Shimura varieties, although we are not aware of any reference where this statement is proved. We show this containment in the case of general Hodge-type Shimura varieties. More generally, we may consider any projective $R$-scheme $X$ endowed with the following structure:
\begin{assumption}\label{assume-X} \ 
\begin{assertionlist}
\item There is a connected, reductive $\ZZ[1/N]$-group $\Gcal$ and a cocharacter $\mu\colon \GG_{\textrm{m},R}\to \Gcal_R$ satisfying the following condition: For $p$ sufficiently large, there exists a smooth map $\zeta_p\colon X_p \to \GpZip^\mu$, where $X_p\colonequals X\otimes_R \overline{\FF}_p$ and  $G_p\colonequals\Gcal \otimes_{\ZZ[1/N]} \FF_p$. Furthermore, $\zeta_p$ is surjective on each connected component of $X_p$.
\item There is a family of vector bundles $(\Vcal_I(\lambda))_{\lambda\in X^*(\mathbf{T})}$ on $X$ such that the restriction of $\Vcal_I(\lambda)$ to $X_p$ coincides with the pullback via $\zeta$ of the vector bundle $\Vcal_I(\lambda)$ on $\GpZip^\mu$.
\end{assertionlist}
\end{assumption}
As we explained in section \ref{sec-tor}, the scheme $\Sscr_{K}^{\Sigma}$ satisfies Assumption \ref{assume-X}. For such a scheme $X$, define similarly $C_X(F)$ as the set of $\lambda\in X^*(\mathbf{T})$ such that $H^0(X\otimes_R F,\Vcal_I(\lambda))\neq 0$. In Theorem \ref{thm-CKC} below, we prove the following:
\begin{theorem}\label{thm-CX-CGS}
We have $C_{X}(\CC)\subset C_{\GS}$.
\end{theorem}
In particular, we may take $X$ to be $\Sscr_K$, which implies that $C_K(\CC)\subset C_{\GS}$. As the setting suggests, our proof relies entirely on characteristic $p$ methods rather than studying the space $H^0(\Sh_K(\mathbf{G},\mathbf{X}),\Vcal_I(\lambda))$ directly via the theory of automorphic representations or Lie algebra cohomology.

\subsubsection{The zip cone}\label{subsec-CZip}

We now consider the case $F=\overline{\FF}_p$. In our approach, the proof of Theorem \ref{thm-CX-CGS} relies on the study of $C_X(\overline{\FF}_p)$ for various prime numbers $p$. In \cite{Goldring-Koskivirta-global-sections-compositio, Goldring-Koskivirta-divisibility}, the authors started a vast project to investigate the set $C_K(\overline{\FF}_p)$ using the stack of $G$-zips. For a cocharacter datum $(G,\mu)$ over $\FF_q$, we defined the zip cone of $(G,\mu)$ in \cite[\S 1.2]{Koskivirta-automforms-GZip} and \cite[\S3]{Goldring-Imai-Koskivirta-weights} as
\begin{equation}
    C_{\zip}\colonequals \{\lambda\in X^*(T) \mid H^0(\GZip^\mu, \Vcal_I(\lambda))\neq 0\}.
\end{equation}
This cone can be seen as a group-theoretical version of the set $C_K(\overline{\FF}_p)$ in the case of Shimura varieties. To emphasize the analogy between $S_K$ and $\GZip^\mu$, we call $H^0(\GZip^\mu, \Vcal_I(\lambda))$ the space of automorphic forms of weight $\lambda$ on $\GZip^\mu$. Since $V_I(\lambda)=0$ when $\lambda$ is not $L$-dominant, $C_{\zip}$ is a subset of the set of $L$-dominant characters $X_{+,I}^*(T)$. One can see that $C_{\zip}$ is a subcone of $X^*(T)$ (\cite[Lemma 1.4.1]{Koskivirta-automforms-GZip}). For a cone $C\subset X^*(T)$, define the saturated cone $\langle C \rangle$ as:
\begin{equation}
 \langle C \rangle \colonequals \{\lambda\in X^*(T) \mid \exists N\geq 1, N\lambda \in C\}.   
\end{equation}
We say that $C$ is saturated in $X^*(T)$ if $\langle C\rangle =C$. We explain the main conjecture that motivates the series of papers \cite{Goldring-Koskivirta-global-sections-compositio, Goldring-Imai-Koskivirta-weights, Goldring-Koskivirta-divisibility}. Consider the special fiber $S_K$ of good reduction of a Hodge-type Shimura variety (such that $\dim(S_K)>1$ or $S_K=S_K^{\Sigma}$), and its associated map $\zeta\colon S_K\to \GZip^\mu$. Since $\zeta$ is surjective, we have a natural inclusion
\begin{equation}
    H^0(\GZip^\mu,\Vcal_I(\lambda)) \subset H^0(S_K,\Vcal_I(\lambda)).
\end{equation}
In particular, we deduce $C_{\zip}\subset C_K(\overline{\FF}_p)$.

\begin{conjecture}\label{conj-Sk}
One has $\langle C_K(\overline{\FF}_p) \rangle = \langle C_{\zip} \rangle$.
\end{conjecture}
It was noted in \cite[Corollary 1.5.3]{Koskivirta-automforms-GZip} that the set $\langle C_K(\overline{\FF}_p) \rangle$ is independent of the level (because the change of level maps are finite etale). Therefore, the above conjecture is indeed reasonable. However, note that the set $C_K(\overline{\FF}_p)$ highly depends on the choice of the level $K$.

More generally, we expect Conjecture \ref{conj-Sk} to hold for any scheme $X$ endowed with a map $\zeta\colon X\to \GZip^\mu$ satisfying the conditions of \cite[Conjecture 2.1.6]{Goldring-Koskivirta-global-sections-compositio}. In particular, it should hold when $X$ is proper and irreducible and $\zeta$ is smooth, surjective (it may also be possible to remove the assumption that $\zeta$ is smooth). As explained in the introduction, we have $C_K(\CC)\subset C_K(\overline{\FF}_p)$. Furthermore, it is expected that $\langle C_K(\CC) \rangle = C_{\GS}$. Hence, Conjecture \ref{conj-Sk} predicts the containment $C_{\GS}\subset \langle C_{\zip} \rangle$ (which is a purely group-theoretical statement). In \cite[Theorem 6.4.2]{Goldring-Imai-Koskivirta-weights}, we prove $C_{\GS}\subset \langle C_{\zip} \rangle$ for any arbitrary pair $(G,\mu)$, which gives evidence for Conjecture \ref{conj-Sk}.

\section{Automorphic forms in characteristic $p$}

We first work a fixed prime $p$ in sections \ref{subsec-stack-zip-flag}, \ref{subsec-Hasse-sections}, \ref{subsec-cones-flag}. In section \ref{subsec-vanishp}, we consider objects in families and let $p$ go to infinity.

\subsection{Notation} \label{subsec-not}
For now, fix a cocharacter datum $(G,\mu)$ over $\FF_q$, i.e $G$ is a connected, reductive group over $\FF_q$ and $\mu\colon \GG_{\textrm{m},k}\to G_k$ is a cocharacter, where $k$ is an algebraic closure of $\FF_q$. Let $(G,P,Q,L,M,\varphi)$ be the attached zip datum. For simplicity, assume that there is an $\FF_q$-Borel pair $(B,T)$ such that $\mu$ factors through $T$ and $B\subset P$ (this can always be achieved after possibly changing $\mu$ to a conjugate cocharacter). Then, the group $\Gal(k/\FF_q)$ acts naturally on $X^*(T)$. Let $W=W(G_k,T)$ be the Weyl group of $G_k$. Similarly, $\Gal(k/\FF_q)$ acts on $W$ and the actions of $\Gal(k/\FF_q)$ and $W$ on $X^*(T)$ and $X_*(T)$ are compatible in a natural sense. For $\alpha \in \Phi$, let $s_\alpha \in W$ be the corresponding reflection. The system $(W,\{s_\alpha \mid \alpha \in \Delta\})$ is a Coxeter system. 
We write $\ell  \colon W\to \NN$ for the length function, and $\leq$ for the Bruhat order on $W$. Let $w_0$ denote the longest element of $W$. For a subset $K\subset \Delta$, let $W_K$ denote the subgroup of $W$ generated by $\{s_\alpha \mid \alpha \in K\}$. Write $w_{0,K}$ for the longest element in $W_K$. Let ${}^KW$ (resp. $W^K$) denote the subset of elements $w\in W$ which have minimal length in the coset $W_K w$ (resp. $wW_K$). Then ${}^K W$ (resp. $W^K$) is a set of representatives of $W_K\backslash W$ (resp. $W/W_K$). The map $g\mapsto g^{-1}$ induces a bijection ${}^K W\to W^K$. The longest element in the set ${}^K W$ (resp. $W^K$) is $w_{0,K} w_0$ (resp. $w_0 w_{0,K}$). For any parabolic $P'\subset G_k$ containing $B$, write $I_{P'}\subset \Delta$ for the type of $P'$, i.e. the subset of simple roots of the unique Levi subgroup of $P'$ containing $T$. For an arbitrary parabolic $P'\subset G_k$, let $I_{P'}$ be the type of the unique conjugate of $P'$ containing $B$. Put $I\colonequals I_P$ and $J\colonequals I_Q$. We set
\begin{equation}
    z=\sigma(w_{0,I})w_0=w_0w_{0,J}.
\end{equation}
The triple $(B,T,z)$ is a $W$-frame, in the terminology of \cite[Definition 2.3.1]{Goldring-Koskivirta-zip-flags} (we will simply call such a triple a frame). In sections \ref{subsec-stack-zip-flag}, \ref{subsec-Hasse-sections}, \ref{subsec-cones-flag}, we let $X$ be a projective scheme over $k=\overline{\FF}_p$ endowed with a map $\zeta\colon X\to \GZip^\mu$ satisfying:
\begin{assumption}\label{assume-atp} \ 
\begin{assertionlist}
\item $\zeta$ is smooth.
\item The restriction of $\zeta$ to any connected component of $X$ is surjective.
\end{assertionlist}
\end{assumption}
For $\lambda\in X^*(T)$, we write again $\Vcal_I(\lambda)$ for the pullback via $\zeta$ of $\Vcal_I(\lambda)$. Write $C_X$ for the set of $\lambda\in X^*(T)$ such that $H^0(X,\Vcal_I(\lambda))\neq 0$.

\subsection{The flag space}\label{subsec-stack-zip-flag}
The rank of the vector bundle $\Vcal_I(\lambda)$ equals the dimension of the representation $V_I(\lambda)$, which can be very large. For this reason, it convenient to consider line bundles on the flag space of $X$ and of $\GZip^\mu$ instead. We recall the definitions below.

\subsubsection{The stack of zip flags}
The stack of zip flags (\cite[Definition 2.1.1]{Goldring-Koskivirta-Strata-Hasse}) is defined as
\begin{equation}\label{eq-Gzipflag-PmodB}
\GF^\mu=[E\backslash (G_k \times P/B)]    
\end{equation}
where the group $E$ acts on the variety $G_k \times (P/B)$ by the rule $(a,b)\cdot (g,hB) \colonequals (agb^{-1},ahB)$ for all $(a,b)\in E$ and all $(g,hB)\in G_k \times P/B$. The first projection $G_k \times P/B \to G_k$ is $E$-equivariant, and yields a natural morphism of stacks
 \begin{equation}\label{projmap-flag}
     \pi  \colon  \GF^\mu \to \GZip^\mu
 \end{equation}
whose fibers are isomorphic to $P/B$. Set $E' \colonequals E\cap (B\times G_k)$. The injective map $G_k \to G_k \times P/B;\ g\mapsto (g,B)$ induces an isomorphism of stacks $[E' \backslash G_k]\simeq \GF^\mu$ (see \cite[(2.1.5)]{Goldring-Koskivirta-Strata-Hasse}).  

\subsubsection{Line bundles $\Vcal_{\flag}(\lambda)$}
To any character $\lambda\in X^*(T)$, we can naturally attach a line bundle $\Vcal_{\flag}(\lambda)$ on $\GF^\mu$. Indeed, we may view $\lambda$ as a character of $E'$ via the first projection $E'\to B$ and use the associated sheaf construction for the quotient stack $[E'\backslash G_k]$. We have by \cite[Proposition 3.2.1]{Imai-Koskivirta-partial-Hasse}:
\begin{equation}\label{pushforward-lambda}
    \pi_*(\Vcal_{\flag}(\lambda))=\Vcal_I(\lambda).
\end{equation}
In particular, we have an identification
\begin{equation}\label{ident-H0-GF}
    H^0(\GZip^\mu,\Vcal_I(\lambda))=H^0(\GF^\mu,\Vcal_{\flag}(\lambda)).
\end{equation}
The line bundles $\Vcal_{\flag}(\lambda)$ satisfy the following identity:
\begin{equation}\label{Vflag-additivity}
\Vcal_{\flag}(\lambda+\lambda')=\Vcal_{\flag}(\lambda)\otimes \Vcal_{\flag}(\lambda'), \quad \forall \lambda, \lambda'\in X^*(T).
\end{equation}
In particular, this identity combined with the identification \eqref{ident-H0-GF} shows that $C_{\zip}$ is stable by sum, hence is indeed a subcone of $X^*(T)$.

\subsubsection{Flag stratification}
Another important feature of $\GF^\mu$ is that it carries a locally closed stratification $(\Fcal_w)_{w\in W}$. First, define the Schubert stack as the quotient stack
\begin{equation}\label{equ-def-Sbt}
\Sbt \colonequals [B\backslash G_k /B].
\end{equation}
The underlying topological space of $\Sbt$ is homeomorphic to $W$, endowed with the topology induced by the Bruhat order on $W$. This follows easily from the Bruhat decomposition of $G$. There is a smooth, surjective map of stacks
 \begin{equation}\label{eq-GF-to-Sbt}
     \psi  \colon  \GF^\mu \to \Sbt.
 \end{equation}
It is defined as follows: Since the group $E'$ is contained in $B\times {}^zB$, we have a natural projection map $[E'\backslash G_k]\to [B\backslash G_k/{}^z\! B]$. Composing this map with the isomorphism $[B\backslash G_k/{}^z B]\to [B\backslash G_k/B]$ induced by $G_k\to G_k;\ g\mapsto gz$, we obtain the map $\psi$ in \eqref{eq-GF-to-Sbt}. For $w\in W$, put $\Sbt_w\colonequals [B\backslash BwB /B]$, it is a locally closed substack of $\Sbt$. The flag strata of $\GF^\mu$ are defined as the fibers of $\psi$. Specifically, for $w\in W$ put:
\begin{equation}
F_w \colonequals B(wz^{-1}){}^zB=BwBz^{-1}.
\end{equation} 
Then $F_w$ is locally closed in $G_k$ of dimension $\dim(F_w)=\ell(w)+\dim(B)$. Via the isomorphism $\GF^\mu\simeq [E'\backslash G_k]$, the flag strata of $\GF^\mu$ are the locally closed substacks
\begin{equation}\label{zipflag-Cw}
\Fcal_w \colonequals [E'\backslash F_w], \quad w\in W.    
\end{equation}
The set $F_{w_0}\subset G_k$ is open in $G_k$ and similarly the stratum $\Fcal_{w_0}$ is open in $\GF^\mu$. The Zariski closure $\overline{F}_w$ is normal by \cite[Theorem 3]{Ramanan-Ramanathan-projective-normality} and coincides with $\bigcup_{w'\leq w}F_{w'}$.

\subsubsection{The flag space of $X$}\label{subsec-flagX}
Define the flag space $Y\colonequals \Flag(X)$ of $X$ as the fiber product
$$\xymatrix@1@M=5pt{
\Flag(X)\ar[r]^-{\zeta_{\flag}} \ar[d]_{\pi_X} & \GF^\mu \ar[d]^{\pi} \\
X \ar[r]_-{\zeta} & \GZip^\mu.
}$$
For $w\in W$, put $Y_w\colonequals \zeta_{\flag}^{-1}(\Fcal_w)$. We obtain on $Y$ a similar stratification by locally closed, smooth subschemes. For $\lambda\in X^*(T)$, we denote again by $\Vcal_{\flag}(\lambda)$ the pullback of the line bundle $\Vcal_{\flag}(\lambda)$ via $\zeta_{\flag}$. Similarly to $\GZip^\mu$, we have the formula $\pi_{X,*}(\Vcal_{\flag}(\lambda))=\Vcal_I(\lambda)$. In particular, we have an identification
\begin{equation} \label{identif-lambda}
    H^0(X,\Vcal_I(\lambda))=H^0(Y,\Vcal_{\flag}(\lambda)).
\end{equation}

\subsection{Hasse cones of flag strata}\label{subsec-Hasse-sections}

To a pair of characters $(\lambda, \nu)\in X^*(T)\times X^*(T)$, we can attach a line bundle $\Vcal_{\Sbt}(\lambda,\nu)$ on the stack $\Sbt$, as in \cite[I.2.2]{Goldring-Koskivirta-Strata-Hasse} (where it was denoted by $\Lcal_{\Sbt}(\lambda,\nu)$). For each $w\in W$, the space $H^0(\Sbt_w,\Vcal_{\Sbt}(\lambda,\nu))$ has dimension $\leq 1$ and is nonzero if and only if $\nu=-w^{-1}\lambda$ (\loccitn, Theorem 2.2.1). For each $w\in W$ and $\lambda\in X^*(T)$, denote by $f_{w,\lambda}$ a nonzero element of the one-dimensional space $H^0(\Sch_w,\Vcal_{\Sch}(\lambda,-w^{-1} \lambda))$. Put
\begin{equation}\label{Ew-def}
    E_w\colonequals \{\alpha\in \Phi_+ \ \mid \ ws_\alpha<w \ \textrm{and} \ \ell(ws_\alpha)=\ell(w)-1 \}.
\end{equation}
Elements $w'\in W$ such that $w'<w$ and $\ell(w')=\ell(w)-1$ will be called lower neighbours of $w$. They correspond bijectively to the set $E_w$ by the map $\alpha \mapsto ws_\alpha$. Define $X^*_{+,w}(T)\subset X^*(T)$ as the subset of $\chi\in X^*(T)$ such that $\langle \chi,\alpha^\vee \rangle \geq 0$ for all $\alpha\in E_w$. Let $\chi\in X^{*}(T)$. By \loccitn, Theorem 2.2.1, the multiplicity of $\div(f_{w,-w\chi})$ along $\Sbt_{ws_\alpha}$ is precisely $\langle \chi,\alpha^\vee\rangle$ for all $\alpha\in E_w$. Hence $f_{w,-w\chi}$ extends to the Zariski closure $\overline{\Sbt}_w$ if and only if $\chi\in X^*_{+,w}(T)$. For any $\lambda,\nu \in X^*(T)$, one has the formula
\begin{equation}
\psi^*(\Vcal_{\Sbt}(\lambda,\nu))=\Vcal_{\flag}(\lambda + q w_{0,I}w_0\sigma^{-1}(\nu))    
\end{equation}
by \cite[Lemma 3.1.1 (b)]{Goldring-Koskivirta-Strata-Hasse} (note that \loccit contains a typo; it should be $\sigma^{-1}$ instead of $\sigma$). In particular, the pullback $\psi^*(\Vcal_{\Sbt}(\lambda,-w^{-1}\lambda))$ coincides with $\Vcal_{\flag}(\lambda - q w_{0,I}w_0\sigma^{-1}(w^{-1}\lambda))$. Define a map 
\begin{equation}
    h_w\colon X^*(T)\to X^*(T), \quad \chi\mapsto -w\chi + q w_{0,I}w_0\sigma^{-1}(\chi).
\end{equation}
Hence $\psi^*(\Vcal_{\Sbt}(-w\lambda,\lambda))=\Vcal_{\flag}(h_w(\lambda))$. Note that for any $w\in W$, the map $h_w\colon X^*(T)\to X^*(T)$ induces an automorphism of $X^*(T)_{\QQ}$ (because $h_w\otimes \FF_p$ is clearly an automorphism of $X^*(T)\otimes_\ZZ \FF_p$). For each $\chi\in X^*(T)$, define
\begin{equation}
    \Ha_{w,\chi} \colonequals \psi^*(f_{w,-w\chi}).
\end{equation}
By the above discussion, $\Ha_{w,\chi}$ is a section over the stratum $\Fcal_w$ of the line bundle $\Vcal_{\flag}(h_w(\chi))$ and $\Ha_{w,\chi}$ extends to $\overline{\Fcal}_w$ if and only if $\chi\in X^*_{+,w}(T)$. The multiplicity of $\div(\Ha_{w,\chi})$ along $\overline{\Fcal}_{ws_\alpha}$ is precisely $\langle \chi,\alpha^\vee \rangle$ for all $\alpha\in E_w$. Define the Hasse cone $C_{\Hasse,w}$ by
\begin{equation} \label{CHassew-def}
    C_{\Hasse,w}\colonequals h_w(X^*_{+,w}(T)).
\end{equation}
Concretely, $C_{\Hasse,w}$ is the set of all possible weights $\lambda\in X^*(T)$ of nonzero sections over $\overline{\Fcal}_{w}$ which arise by pullback from $\overline{\Sbt}_w$. 

\subsection{Regularity of strata}
In general, there exist many sections on $\overline{\Fcal}_w$ that do not arise by pullback from $\overline{\Sbt}_w$. For $w\in W$, define the cones $C_{\flag,w}$ and $C_{Y,w}$ as follows:
\begin{align}
    C_{\flag,w}&\colonequals\{\lambda\in X^*(T) \ \mid \ H^0(\overline{\Fcal}_w,\Vcal_{\flag}(\lambda))\neq 0\} \\
    C_{Y,w}&\colonequals \left\{\lambda\in X^*(T) \relmiddle| \ H^0(\overline{Y}_w,\Vcal_{\flag}(\lambda))\neq 0 \right\}.
\end{align}
In particular, via the identification \eqref{identif-lambda}, the cone $C_{Y,w_0}$ is the set of $\lambda\in X^*(T)$ such that $\Vcal_I(\lambda)$ admits nonzero sections over $X$, hence we have an equality $C_{Y,w_0}=C_X$ and similarly $C_{\flag,w_0}=C_{\zip}$. For any $w\in W$, we clearly have
\begin{equation}
    C_{\Hasse,w}\subset C_{\flag,w}\subset C_{Y,w}.
\end{equation}

\begin{definition} \label{def-regular} Let $w\in W$.
\begin{definitionlist}
\item We say that $Y_w$ is Hasse-regular if $\langle C_{Y,w} \rangle = \langle C_{\Hasse, w} \rangle$.
\item We say that $Y_w$ is flag-regular if $\langle C_{Y,w} \rangle = \langle C_{\flag, w}\rangle$.
\end{definitionlist}
\end{definition}
A Hasse-regular stratum is obviously flag-regular. Assumptions \ref{assume-atp} are made so that the following easy lemma holds:
\begin{lemma}[{\cite[Proposition 3.2.1]{Goldring-Koskivirta-global-sections-compositio}}] \label{lw1-reg}
If $\ell(w)=1$, then $Y_w$ is Hasse-regular.
\end{lemma}
Since $C_{Y,w_0}=C_X$ and $C_{\Ycal,w_0}=C_{\zip}$, Conjecture \ref{conj-Sk} asserts that the maximal flag stratum $Y_{w_0}$ is always flag-regular. It is not Hasse-regular in general (but it is conjecturally Hasse-regular for Hasse-type zip data, see \cite{Goldring-Imai-Koskivirta-weights}). In the case of Hilbert--Blumenthal Shimura varieties attached to a totally real extension $\mathbf{F}/\QQ$, a sufficent condition for the Hasse-regularity of strata is given in
\cite[Theorem 4.2.3]{Goldring-Koskivirta-global-sections-compositio}. When $p$ is split in $\mathbf{F}$, all strata are Hasse-regular. For a general prime $p$, the criterion involves the parity of "jumps" in the orbit under the Galois action. A more elegant proof, using the notion of "intersection cone" (introduced in \cite{Goldring-Koskivirta-divisibility}) can be found in the unpublished note \cite{Hilb-Blum-via-IC}.

Let $w\in W$ with $\ell(w)=1$, and write $w=s_\beta$ with $\beta\in \Delta$. One checks readily:
\[\langle C_{\Hasse,w} \rangle = \{\lambda\in X^*(T) \ \mid \ \langle h_{w}^{-1}(\lambda), \beta^\vee\rangle \geq 0 \}.\]
We deduce:
\begin{proposition}\label{prop-strata1}
Let $f\in H^0(Y,\Vcal_{\flag}(\lambda))$ such that the restriction of $f$ to the stratum $Y_w$ is not identically zero, where $w=s_\beta$ ($\beta\in \Delta$). Then we have $\langle h_{s_{\beta}}^{-1}(\lambda), \beta^\vee\rangle \geq 0$.
\end{proposition}

\subsection{Upper bounds for strata cones} \label{subsec-cones-flag}

\subsubsection{Intersection cones} \label{subsec-inter-cone}
We recall the notion of intersection cone introduced in \cite{Goldring-Koskivirta-divisibility}, which will be used in section \ref{sec-Shim-van}. We give a simplified version of the one appearing in \loccit which suffices for our purpose. 

\begin{definition}\label{sep-syst-E}
For each $w\in W$, let $\EE_w\subset E_w$ be a subset (possibly empty) and let $\{\chi_{\alpha}\}_{\alpha\in \EE_w}$ be a family of characters satisfying the conditions:
\begin{definitionlist}
\item $\langle \chi_\alpha,\alpha^\vee\rangle >0$,
\item $\langle \chi_\alpha,\beta^\vee\rangle =0$ for all $\beta\in E_w\setminus\{\alpha\}$.
\end{definitionlist}
We call $\EE=(\EE_w)_{w\in W}$ a separating system. 
\end{definition}

We fix such a system $\EE$ and define the intersection cones $(C^{+,\EE}_w)_{w\in W}$ of $\EE$ as follows. First, set
\begin{align*}
 \Gamma_{w}&\colonequals\sum_{\alpha\in \EE_w}\ZZ_{\geq 0}\chi_{\alpha}  \\
C_{\Hasse,w}^{\EE} &\colonequals h_w(\Gamma_w).
\end{align*}
Note that $\chi_\alpha \in X_{+,w}^*(T)$, therefore $\Gamma_w\subset C_{\Hasse,w}$, but $\Gamma_w$ can be much smaller (for example, if we choose $\EE_w$ to be a singleton, $\Gamma_w$ is a half-line in $X^*(T)$.

\begin{definition}\label{def-intersumcone}
For $\ell(w)=1$, set $C^{+,\EE}_{w}\colonequals C_{\Hasse,w}$. For $\ell(w)\geq 2$, define inductively
\begin{equation}
 C^{+,\EE}_{w}\colonequals C^{\EE}_{\Hasse,w}+ \bigcap_{\alpha\in \EE_w} C^{+,\EE}_{ws_\alpha}.
\end{equation}
In the case $\EE_w=\emptyset$, we define by convention $\bigcap_{\alpha\in \EE_w} C^{+,\EE}_{ws_\alpha}=X^*(T)$.
\end{definition}
The intersection cones provide upper bounds for the strata cones $C_{Y,w}$. Specifically, by \cite[Theorem 2.3.9]{Goldring-Koskivirta-divisibility}, we have:
\begin{theorem}\label{thm-sep-syst}
Let $\EE$ be a separating system. For each $w\in W$, we have
\begin{equation}
 C_{Y,w} \subset \langle C^{+,\EE}_{w} \rangle.
\end{equation}
\end{theorem}

\subsubsection{Upper bound by degree}

In general, we do not know a way to construct nontrivial separating systems $\EE$ for arbitrary reductive groups. For a given $w\in W$ and $\alpha\in E_w$, there may not always exist a character $\chi_\alpha$ satisfying the conditions explained in section \ref{subsec-inter-cone}. Here, we explain a more straightforward method to produce an upper bound for $C_{Y,w}$. The advantage of this method is that it applies in general. However, it only gives a rather coarse upper bound (but it will be sufficient for our purpose).

Since $h_w\colon X^*(T)_\QQ \to X^*(T)_\QQ$ is an automorphism, there exists $N\geq 1$ such that $N X^*(T)\subset h_w(X^*(T))$. We fix such an integer. For $\lambda\in X^*(T)$, let $\chi_{w,\lambda}\colonequals h_w^{-1}(N\lambda)$ and write $\Ha_w^\lambda\colonequals \Ha_{w,\chi_{w,\lambda}}$ for the associated Hasse section on $\Fcal_w$ and $Y_w$, with weight $N\lambda$. Since the map $\zeta_{\flag}\colon Y\to \GF^\mu$ is smooth and surjective, the multiplicities of sections do not change under pullback. Hence, the divisor of $\Ha_{w}^\lambda$ over $Y_w$ is given by:
\[\div(\Ha_{w}^\lambda) =\sum_{\alpha\in E_w} \langle \chi_{w,\lambda},\alpha^\vee \rangle [\overline{Y}_{w s_\alpha}]. \]
Define
\[\deg(w,\lambda) \colonequals \frac{1}{N}\deg(\div(\Ha_{w}^\lambda))= \sum_{\alpha\in E_w} \langle h^{-1}_w(\lambda),\alpha^\vee \rangle.\] 
We write $\deg_q(w,\lambda)$ when we want to emphasize that the degree depends on the prime power $q$ (since the map $h_w$ itself depends on $q$). Since $\Ha_{w,\lambda+\lambda'}=\Ha_{w,\lambda} \cdot \Ha_{w,\lambda'}$, we have
\[ \deg(w,\lambda+\lambda') = \deg(w,\lambda) + \deg(w,\lambda'). \]

\begin{lemma}
Let $w\in W$ of length $\geq 1$. Suppose that the space $H^0(\overline{Y}_w,\Vcal_{\flag}(\lambda))$ is nonzero. Then we have $\deg(w,\lambda)\geq 0$.
\end{lemma}

\begin{proof}
Let $f$ be a nonzero section on $\overline{Y}_w$ of weight $\lambda$. Then $f^N/\Ha^\lambda_{w}$ is a rational section of $\Ocal_Y$ over $Y_w$. Since $\overline{Y}_w$ is projective and normal, we have $\deg(\div(f^N/\Ha^\lambda_{w}))=0$, hence $\deg(\div(f))=\frac{1}{N}\deg(\div(\Ha^\lambda_{w}))=\deg(w,\lambda)$. Since $\div(f)$ is effective, the result follows.
\end{proof}
Define $C^{\deg}_w\colonequals \{\lambda\in X^*(T) \mid \deg(w,\lambda)\geq 0\}$. As a consequence, we deduce:
\begin{corollary}
We have $\langle C_{Y,w} \rangle \subset C^{\deg}_w$.
\end{corollary}
In other words, if $\deg(w,\lambda)< 0$, then the space $H^0(\overline{Y}_w,\Vcal_{\flag}(\lambda))$ is zero. We will apply this result when $p=q$ tends to infinity. Therefore, we need to know the behaviour of the function $\deg_q(w,\lambda)$ as $q$ varies. By \cite[Lemma 3.1.3]{Goldring-Koskivirta-Strata-Hasse}, $h_w^{-1}(\lambda)$ is an expression of the form $-\frac{1}{q^m-1} \sum_{i=0}^{m-1} q^i u_i \sigma^{i}(\lambda)$ for certain elements $u_i\in W$ independent of $q$. Furthermore, for $i=m-1$, the element $u_{m-1}\sigma^{m-1}(\lambda)$ equals $\sigma(w_0 w_{0,I}\lambda)$. We deduce:

\begin{proposition} \label{prop-degq-pol}
There exists an integer $m\geq 1$ such that
\[\deg_q(w,\lambda) = \frac{1}{q^m-1} \left(q^{m-1} \sum_{\alpha\in E_w } \langle \sigma(w_{0,I}w_0 \lambda),\alpha^\vee\rangle \ + \ \textnormal{lower terms} \right) \]
\end{proposition}

\subsection{Vanishing in families} \label{subsec-vanishp}

In this section we take $X$ to be a scheme over $R$ satisfying Assumption \ref{assume-X} (for example $X=\Sscr_K$). By flat base change along the map $\spec(\CC)\to \spec(R)$, we have $H^0(X\otimes_R \CC,\Vcal_I(\lambda)) = H^0(X,\Vcal_I(\lambda))\otimes_R \CC$. Hence, for $\lambda\in C_X(\CC)$ the space $H^0(X,\Vcal_I(\lambda))$ is also nonzero. Therefore, we can apply the proof of \cite[Proposition 1.8.3]{Koskivirta-automforms-GZip} to show that the space $H^0(X\otimes_R \overline{\FF}_p,\Vcal_I(\lambda))$ is also nonzero for all $p$. In particular, we deduce:
\begin{equation}\label{contained-CXC-CXKp}
    C_X(\CC)\subset C_X(\overline{\FF}_p)
\end{equation}
for all primes where $X_p$ is defined. The main goal of this section is to show $C_X(\CC)\subset C_{\GS}$. We may interpret this as a vanishing result (the space $H^0(X\otimes_R,\Vcal_I(\lambda))$ vanishes for $\lambda$ outside of $C_{\GS}$). We will see later some stronger forms of vanishing results at fixed prime $p$.

Let $f$ be a nonzero section of $\Vcal_I(\lambda)$ over $X$. We will show that the weight $\lambda$ is in $C_{\GS}$ by exploiting the fact that $f$ gives rise to a family $(f_p)_p$, where $f_p$ is the reduction of $f$ to the subscheme $X_p=X\otimes_R\overline{\FF}_p$. For sufficently large $p$, we have by assumption a map $\zeta_p\colon X_p\to \GpZip^\mu$. Denote by $Y_p$ the flag space of $X_p$ as in \ref{subsec-flagX}.

\begin{theorem}\label{thm-plarge-nonzero}
For sufficiently large $p$, the section $f_p$ restricts to a nonzero section on each flag stratum $Y_{p,w}$ (for $w\in W$).
\end{theorem}

\begin{proof}
Clearly, it suffices to show that $f_p$ restricts to a nonzero section on the zero-dimensional stratum for sufficiently large $p$. For this, we will prove by decreasing induction that for each $0\leq i \leq \ell(w_0)$, there exists an element $w_i$ of of length $i$ in $W$ such that $f_p$ is not identically zero on $S_{w_i}$ for sufficiently large $p$. The result is clear for $i=\ell(w_0)$. Suppose that $f_p$ is nonzero on $S_{w_i}$ for large $p$. For a contradiction, assume that $f_p$ is zero on each stratum in the closure of $S_{w_i}$ for infinitely many primes $p$. Choose any character $\chi\in X^*(T)$ such that $\langle \chi,\alpha^\vee \rangle \geq 0$ for all $\alpha\in E_{w_i}$ and $\langle \chi,\alpha_0^\vee\rangle >0$ for at least one $\alpha_0\in E_{w_i}$. The multiplicities of the divisor of $\Ha_{w_i,\chi}$ are the numbers $\langle \chi,\alpha^\vee \rangle$ (for $\alpha\in E_{w_i}$). Hence, by assumption we can find an integer $m$ (independent of $p$) such that for infinitely many primes $p$, the section $f^m_p$ is divisible by $\Ha_{w_i,\chi}$. Thus, we deduce that for infinitely many primes $p$,
\[ \deg_p(w_i,m\lambda - h_{w_i,p}(\chi)) = m \deg_p(w_i,\lambda) - \sum_{\alpha \in E_w} \langle \chi,\alpha^\vee \rangle \geq 0 \]
When $p$ tends to infinity, the expression $\deg_p(w_i,\lambda)$ tends to zero by Proposition \ref{prop-degq-pol}. Since $\langle \chi,\alpha_0^\vee\rangle >0$ for at least one $\alpha_0\in E_{w_i}$, we have a contradiction. The result follows.
\end{proof}

\begin{rmk}
In this remark, we consider the case $X=\Sscr_K$. Theorem \ref{thm-plarge-nonzero} is related to Deuring's theorem regarding the superspecial reduction of abelian varieties. Indeed, assume the following result: any CM abelian variety over $\overline{\QQ}$ has superspecial reduction for infinitely many primes $p$. Then a slightly weaker variant of Theorem \ref{thm-plarge-nonzero} would follow immediately (at least for the Siegel-type Shimura variety $\Acal_g$) as follows: Since CM points are dense, we may choose a CM point $x\in \Sscr_K(\overline{\QQ})$ such that $f(x)\neq 0$. Then, for all $p$ sufficently large, we must have $f_p(x_p)\neq 0$ where $x_p$ denotes the specialization of $x$ (which is well-defined for large $p$). Since $x_p$ lies in the zero-dimensional stratum for infinitely many primes, $f_p$ is nonzero on the zero-dimensional stratum (hence on all strata) for inifitely many primes $p$. This is slightly weaker than the content of Theorem \ref{thm-plarge-nonzero}, which states the same result for sufficiently large $p$.
\end{rmk}

\begin{proposition}\label{prop-van-length1}
Let $f\in H^0(X,\Vcal_{I}(\lambda))$. Suppose that for infinitely many primes $p$, the section $f_p$ (viewed as a section of $\Vcal_{\flag}(\lambda)$ on the flag space $Y_p$) restricts to a nonzero section on each flag stratum of $Y_p$ of length one. Then $\lambda\in C_{\GS}$. 
\end{proposition}

\begin{proof}
By Proposition \ref{prop-strata1}, we have $\langle h_{s_{\beta},p}^{-1}(\lambda), \beta^\vee\rangle \geq 0$ for all $\beta\in \Delta$ and infinitely many primes $p$. Looking at the leading term, we obtain $\langle \sigma(w_{0,I} w_{0}\lambda),\beta^\vee \rangle \geq 0$ for all $\beta\in \Delta$. Since $\sigma(\Delta)=\Delta$, we deduce that $w_{0,I} w_{0} \lambda$ is a dominant character. In other words, $\lambda\in C_{\GS}$. 
\end{proof}

We deduce immediately from Theorem \ref{thm-plarge-nonzero} and Proposition \ref{prop-van-length1} our main result of this section:
\begin{theorem}\label{thm-CKC}
We have $C_X(\CC)\subset C_{\GS}$.
\end{theorem}

In particular for $X=\Sscr_K$, we obtain $C_K(\CC)\subset C_{\GS}$. We now explain a slighty more precise result.
\begin{definition}
We say that a family of cones $(C_p)_p$ (defined for sufficiently large primes $p$) is asymptotic to $C_{\GS}$ if
\[\bigcap_{p\geq N}  C_p  = C_{\GS}\]
for any $N\geq 1$. We say that $(C_p)_p$ is asymptotically contained in $C_{\GS}$ if $\bigcap_{p\geq N} C_p \subset C_{\GS}$ for all $N\geq 1$.
\end{definition}
The proof of Theorem \ref{thm-CKC} actually shows the following:
\begin{corollary}\label{cor-asymp-0}
The family of cones $(C_{X,p})_p$ is asymptotically contained in $C_{\GS}$.
\end{corollary}
\begin{proof}
Let $\lambda\in \bigcap_{p\geq N} C_{X,p}$. For sufficiently large $p$, there exists a nonzero form $f_p$ over $Y_p$ of weight $\lambda$. Then, we may apply the proof of Theorem \ref{thm-plarge-nonzero} to the family $(f_p)_p$ (even if this family does not arise by reduction from a characteristic zero section). It shows that $\lambda\in C_{\GS}$. The result follows. 
\end{proof}

However, we were not able to show in general that the family of saturated cones $(\langle C_{K,p}\rangle )_p$ is asymptotically contained in $C_{\GS}$. Corollary \ref{cor-asymp-0} is slightly more precise than Theorem \ref{thm-sep-syst}, since it implies $C_X(\CC)\subset \bigcap_{p} C_{X,p}\subset C_{\GS}$ using \eqref{contained-CXC-CXKp}. The proof of Theorem \ref{thm-CKC} explained above crucially uses the fact that we have a family of schemes $(X_p)_p$ for almost all prime numbers $p$. However, the proof gives no information about the set $C_X(\overline{\FF}_p)$ for a fixed prime number $p$. Eventually, we are interested in vanishing results for automorphic forms in both characteristics. Therefore, a more desirable method of proof of Theorem \ref{thm-CKC} is the following: Assume that for each $p$, we can show that any weight $\lambda \in C_{K,p}\colonequals C_{K}(\overline{\FF}_p)$ satisfies certain inequalities
\begin{equation}\label{gamma-ineq}
    \gamma_i(p,\lambda) \leq 0, \quad i=1,\dots N. 
\end{equation}
where $\gamma_i(p,\lambda)$ is an algebraic expression involving $p$ and which is linear in $\lambda$. Denote by $C_{\gamma,p}$ the cone of $\lambda \in X_{+,I}^*(T)$ satisfying the inequalities \eqref{gamma-ineq}. By assumption, we have $C_{K,p}\subset C_{\gamma,p}$ (note that since $C_{\gamma,p}$ is defined by inequalities, it is obviously saturated, hence we also have $\langle C_{K,p}\rangle \subset C_{\gamma,p}$). We deduce:
\begin{equation}\label{contain-eq}
    C_K(\CC)\subset \bigcap_{p >>0 }  C_{K,p} \subset \bigcap_{p >>0 } C_{\gamma,p}.
\end{equation}
Therefore, if we can choose $(\gamma_i)_{1\leq i \leq N}$ such that $\bigcap_{p >>0 } C_{\gamma,p}=C_{\GS}$, we obtain the desired containment $C_{K}(\CC)\subset C_{\GS}$. We call such a family $(\gamma_i)_{i=1,\dots ,N}$ a GS-approximation of the family $(C_{K,p})_p$. This method of proof gives much more control and information on the weights of automorphic forms in all characteristics. We will implement such a strategy in the next section. In general, it is a difficult problem to give an upper bound for the cone $C_{K,p}$ at a fixed prime $p$, let alone construct a GS-approximation for the family $(C_{K,p})_p$. We will do this for unitary Shimura varieties of signature $(n-1,1)$.

\section{Vanishing results for $\GZip^\mu$}

We investigate the strategy explained in section \ref{subsec-vanishp}. Recall that we work at a fixed prime number $p$ and want to show that there exists certain suitable algebraic expressions $(\gamma_i)_{i=1,\dots ,N}$ satisfying $C_{K,p}  \subset C_{\gamma,p}$. However, we also keep in mind that when $p$ varies, we want the condition $\bigcap_{p >>0 } C_{\gamma,p}=C_{\GS}$ to be satisfied.

Write $C_{\zip,p}$ for the zip cone of $(G_p,\mu_p)$. Since $C_{\GS}\subset \langle C_{\zip,p}\rangle \subset \langle C_{K,p}\rangle$, the family $(\gamma_i)_{i=1,\dots ,N}$ would also be a GS-approximation of the family $(C_{\zip,p})_p$. For this reason, we first seek a GS-approximation of the family $(C_{\zip,p})_p$ to gain intuition, which is a more tractable, group-theoretical object. We will give a natural and explicit GS-approximation of $(C_{\zip,p})_p$ in certain cases (including all cases when $G$ is split over $\FF_p$). In the unitary split case of signature $(n-1,1)$, we show in section \ref{sec-Shim-van} that this also provides a GS-approximation of the Shimura cone family $(C_{K,p})_p$.

\subsection{Group-theoretical preliminaries}\label{subsec-gp-prelim}

Let $(G,\mu)$ be a cocharacter datum over $\FF_q$ (as usual, we take $q=p$ for Shimura varieties). Let $\Zcal=(G,P,Q,L,M,\varphi)$ be the attached zip datum (see section \ref{subsec-GZip}). Choose a frame $(B,T,z)$, with $(B,T)$ defined over $\FF_q$ as in section \ref{subsec-not} and $z=\sigma(w_{0,I}) w_0$. Define $B_M\colonequals B\cap M$. We first explain that we can naturally inject the space of global sections $H^0(\GZip^\mu,\Vcal_I(\lambda))$ into a space of regular maps $B_M\to \AA^1$ which are eigenfunctions for a certain action of $T$ on $B_M$. We recall some results from \cite{Goldring-Imai-Koskivirta-weights}. Recall that $H^0(\GZip^\mu,\Vcal_I(\lambda))$ identifies with $H^0(\GF^{\mu},\Vcal_{\flag}(\lambda))$ by \eqref{ident-H0-GF}. Furthermore, using the isomorphism $\GF^\mu\simeq [E'\backslash G_k]$ (see section \ref{subsec-stack-zip-flag}), an element of the space $H^0(\GF^{\mu},\Vcal_{\flag}(\lambda))$ can be viewed as a function $f\colon G_k\to \AA_k^1$ satisfying
\begin{equation}\label{equ-Gzipflag-lambda}
    f(agb^{-1})=\lambda(a) f(g), \quad \forall (a,b)\in E', \ \forall g\in G_k.
\end{equation}
Recall that $\GF^{\mu}$ admits a unique open stratum $\Ucal_{\max}=\Fcal_{w_0}$. Write also $U_{\max}\colonequals F_{w_0}=B w_0 Bz^{-1}$ (the $B\times {}^z B$-orbit of $w_0 z^{-1}=\sigma(w_{0,I})^{-1}$). 

\begin{lemma}[{\cite[Lemma 4.2.1]{Goldring-Imai-Koskivirta-weights}}] \label{lemma-Umax} \ 
The map $B_M\to U_{\max}$, $b\mapsto \sigma(w_{0,I})b^{-1}$ induces an isomorphism $[B_M/T]\simeq \Ucal_{\max}$, where $T$ acts on $B_M$ on the right by the action $B_M\times T\to B_M$ given by $(b,t)\mapsto \varphi(t)^{-1} b \sigma(w_{0,I}) t \sigma(w_{0,I})^{-1}$.
\end{lemma}
For $\lambda\in X^*(T)$, let $S(\lambda)$ denote the space of functions $h\colon B_M\to \AA^1$ satisfying
\begin{equation}\label{equ-Slambda}
h(\varphi(t)^{-1}b\sigma(w_{0,I}) t\sigma(w_{0,I})^{-1})=\lambda(t)^{-1} h(b), \quad \forall t\in T, \ \forall b\in B_M.
\end{equation}
\begin{corollary}
The isomorphism from Lemma \ref{lemma-Umax} induces an isomorphism
\begin{equation}
  \vartheta \colon  H^0(\Ucal_{\max},\Vcal_{\flag}(\lambda))\to S(\lambda).
\end{equation}
\end{corollary}
We describe explicitly this isomorphism. Let $f\in H^0(\Ucal_{\max},\Vcal_{\flag}(\lambda))$, viewed as a function $f\colon U_{\max}\to \AA^1$ satisfying \eqref{equ-Gzipflag-lambda}. The corresponding element $\vartheta(f)\in S(\lambda)$ is the function $B_M \to \AA^1;\ b\mapsto f(\sigma(w_{0,I})b^{-1})$. Conversely, if $h\colon B_M\to \AA^1$ is an element of $S(\lambda)$, the function $f=\vartheta^{-1}(h)$ is given by
\begin{equation}\label{equ-relationf}
f(b_1 \sigma(w_{0,I}) b_2^{-1}) = \lambda(b_1) h(\varphi(\theta^P_L(b_1))^{-1} \theta^Q_M(b_2)), \quad (b_1,b_2)\in B\times {}^zB,
\end{equation}
where the functions $\theta_L^P$ and $\theta_M^Q$ were defined in section \ref{subsec-GZip}. By the property of $h$, the function $f$ is well-defined. 

Given a section of $\Vcal_{\flag}(\lambda)$ over $\GF^{\mu}$, we can restrict it to the open substack $\Ucal_{\max}$, and then apply $\vartheta$ to obtain an element of $S(\lambda)$. Hence, we may view $H^0(\GZip^\mu,\Vcal_I(\lambda))$ as a subspace of $S(\lambda)$. In general, it is difficult to determine the image of this map. On the other hand, by the previous discussion, a section $f\in H^0(\GZip^\mu,\Vcal_I(\lambda))$ can be viewed as a regular function $f\colon G\to \AA^1$ satisfying condition \eqref{equ-Gzipflag-lambda}. In particular, $f$ is equivariant under the action of the unipotent subgroup $U\times V\subset E'$. Denote by $S_{\unip}$ the space of such functions:
\begin{equation*}
    S_{\unip} \colonequals\left\{ f\colon G\to \AA^1 \relmiddle| \ f(ugv)=f(g), \ u\in U, \ v\in V \right\}.
\end{equation*}
Hence, we may also view $H^{0}(\GZip^\mu,\Vcal_I(\lambda))$ as a subspace of $S_{\unip}$. The reason for introducing this space is the following: We will see in the next section that elements of $S_{\unip}$ can be conveniently decomposed with respect to the action of $T\times T$ on $G$.

Finally, for $f\in S_{\unip}$, we define $\widetilde{f}\colon B_M\to \AA^1$ by $\widetilde{f}(b)=f(w_{0,M}b^{-1})$. We write again $\vartheta$ for the map $S_{\unip} \to k[B_M]$, $f\mapsto \widetilde{f}$ (this map is not injective in general). By construction, the following diagram is clearly commutative.
\begin{equation}
    \xymatrix@1@M=8pt{
    H^0(\GZip^\mu,\Vcal_I(\lambda)) \ar@{^{(}->}[d] \ar@{^{(}->}[r]^-{\vartheta} & S(\lambda) \ar@{^{(}->}[d] \\
    S_{\unip} \ar[r]_{\vartheta} & k[B_M]
    }
\end{equation}
In particular, we deduce that if we view a nonzero element $f\in H^0(\GZip^\mu,\Vcal(\lambda))$ as an element of $S_{\unip}$ and then apply $\vartheta\colon S_{\unip}\to k[B_M]$, the result is nonzero. This will imply that when we decompose $f$ in $S_{\unip}$ as a sum of $T\times T$-eigenvectors, at least one of the components of $f$ will map via $\vartheta$ to a nonzero element of $k[B_M]$.

Next, we choose coordinates for the Borel subgroup $B_M$ of $M$. This can be accomplished using the following result. For $\alpha\in \Phi$, let $U_\alpha$ be the corresponding $\alpha$-root group. Recall that by our convention, $\alpha \in \Phi_+$ when $U_{\alpha}$ is contained in the opposite Borel $B^+$ to $B$.
\begin{proposition}[{\cite[XXII, Proposition 5.5.1]{SGA3}}] \label{prop-roots-B}
Let $G$ be a reductive group over $k$ and let $(B,T)$ be a Borel pair. Choose a total order on $\Phi_-$. The $k$-morphism
\begin{equation} \label{equ-map-gamma}
\gamma \colon T\times \prod_{\alpha\in \Phi_-}U_\alpha \to G
\end{equation}
defined by taking the product with respect to the chosen order is a closed immersion with image $B$.
\end{proposition}

We apply Proposition \ref{prop-roots-B} to $(M,B_M)$. Choose an order on $\Phi_{M,-}$ and consider the corresponding map $\gamma$ as in \eqref{equ-map-gamma}, with image $B_M$. For a function $h\colon B_M\to \AA^1$, put $P_h \colonequals h\circ \gamma$. For all $\alpha\in \Phi$, choose an isomorphism $u_\alpha\colon \GG_{\mathrm{a}}\to U_\alpha$ so that 
$(u_{\alpha})_{\alpha \in \Phi}$ is a realization in the sense of \cite[8.1.4]{Springer-Linear-Algebraic-Groups-book}. In particular, we have 
\begin{equation}\label{eq:phiconj}
 t u_{\alpha}(x)t^{-1}=u_{\alpha}(\alpha(t)x), \quad \forall x\in \GG_{\mathrm{a}},\  \forall t\in T.
\end{equation}
Via the isomorphism $u_{\alpha}\colon \GG_{a} \to U_{\alpha}$, we can view $P_h$ as a polynomial $P_h \in k[T][(x_\alpha)_{\alpha\in \Phi_{M,-}}]$, where the $x_\alpha$ are indeterminates indexed by $\Phi_{M,-}$. For $m=(m_\alpha)_\alpha \in \NN^{\Phi_{M,-}}$ and $\xi\in X^*(T)$, denote by $P_{m,\xi}$ the monomial
\begin{equation}\label{equ-monomial}
P_{m,\xi}=\lambda(t)\prod_{\alpha \in \Phi_{M,-}} x_{\alpha}^{m_\alpha} \ \in \  k[T][(x_\alpha)_{\alpha\in \Phi_{M,-}}].   
\end{equation}
We can write any element $P$ of $k[T][(x_\alpha)_{\alpha\in \Phi_{M ,-}}]$ as a sum of monomials
\begin{equation}\label{equ-decomp-monomials}
    P=\sum_{i=1}^N c_{i} P_{m_i,\xi_i}
\end{equation}
where for all $1\leq i \leq N$, we have $m_i\in \NN^{\Phi_{M,-}}$, $\xi_i\in X^*(T)$ and $c_{i}\in k$. Furthermore, we may assume that the $(m_i,\xi_i)$ are pairwise distinct. Under this assumption, the expression \eqref{equ-decomp-monomials} is uniquely determined up to permutation of the indices. For $P\in k[T][(x_\alpha)_{\alpha\in \Phi_{M,-}}]$, define $h_P\colon B_L\to \AA^1$ as the function $P\circ \gamma^{-1}$. For $m=(m_\alpha)_\alpha \in \NN^{\Phi_{L,-}}$ and $\xi\in X^*(T)$, define $h_{m,\xi}\colonequals h_{P_{m,\lambda}}$.

Decompose $k[B_M]$ with respect to the action of $T\times T$ on $B_M$:
\[ k[B_M]=\bigoplus_{(\chi_1,\chi_2)} k[B_M]_{\chi_1,\chi_2} \]
where $k[B_M]_{\chi_1,\chi_2}$ is the set of functions $h\colon B_M\to \AA^1$ satisfying $h(t_1 b t_2^{-1})=\chi_1(t_1)\chi_2(t_2)h(b)$ for characters $\chi_1,\chi_2\in X^*(T)$. Put $\lambda(\chi_1,\chi_2)=q \sigma^{-1}\chi_1+\sigma(w_{0,I})\chi_2$. Then we have:
\begin{equation}
    S(\lambda)=\bigoplus_{\lambda(\chi_1,\chi_2)=\lambda} k[B_M]_{\chi_1,\chi_2}.
\end{equation}
It is clear that functions of the form $h_{m,\xi}$ are $T\times T$-eigenfunctions. Lemma \ref{lemma-formula-hmlambda} determines exactly its weight $(\chi_1,\chi_2)$. The proof of the lemma is similar to that of \cite[Lemma 4.3.3]{Goldring-Imai-Koskivirta-weights}.

\begin{lemma}\label{lemma-formula-hmlambda}
Let $(m,\xi)\in \NN^{\Phi_{M,-}}\times X^*(T)$. Then $h_{m,\xi}$ lies in $k[B_M]_{\chi_1,\chi_2}$ for 
\begin{equation*}
    \chi_1=\xi, \quad \textrm{and} \quad \chi_2=-\xi+\sum_{\alpha\in \Phi_{M,-}} m_\alpha \alpha.
\end{equation*}
\end{lemma}
For $(m,\xi)\in \NN^{\Phi_{M,-}}\times X^*(T)$, define the weight $\omega(m,\xi)$ by
\begin{equation}\label{omega-weight}
\omega(m,\xi)\colonequals q \sigma^{-1}(\xi)- w_{0,M}\xi +\sum_{\alpha \in \Phi_{M,-}} m_\alpha (w_{0,M} \alpha) \in X^*(T).
\end{equation}
It follows immediately from Lemma \ref{lemma-formula-hmlambda} that the function $h_{m,\xi}\colon B_M\to \AA^1$ lies in $S(\omega(m,\xi))$.

\subsection{Unipotent-invariance cone} \label{sec-unip-cone}

\subsubsection{Regular maps invariant under a unipotent subgroup}
To give an upper bound on the cone $C_{\zip}$, we will view sections over $\GZip^\mu$ as regular functions $f\colon G\to \AA^1$ and use their invariance under the action of the unitary group $U\times V$. We will show that this invariance condition forces the weight of $f$ to be constrained to a certain region of $X^*(T)$. We will therefore call this the "unipotent-invariance cone". 

Let $(G,\mu)$ be a cocharacter datum over $\FF_q$. Let $\Zcal_\mu=(G,P,Q,L,M,\varphi)$ be the attached zip datum. To simplify, we restrict ourselves to the case when $P$ is defined over $\FF_q$. Recall that $U=R_{\textrm{u}}(P)$ and $V=R_{\textrm{u}}(Q)$. The key fact is the following easy lemma:

\begin{lemma}\label{lem-decomp}
Let $f\colon G\to \AA^1$ be a regular function satisfying $f(gu)=f(g)$ for all $g\in G$ and all $u$ in the unipotent radical $U'$ of a standard parabolic $P'\subset G$. Let $I'\subset \Delta$ be the type of $P'$. Then:
\begin{assertionlist}
\item We may decompose $f$ uniquely as $f=\sum_{\chi} f_\chi$ where $\chi\in X^*(T)$, such that $f_\chi$ is also $U'$-equivariant and satisfies furthermore $f_\chi(gt)=\chi(t)^{-1} f_\chi(g)$ for all $g\in G$, $t\in T$.
\item For all $\chi$ such that $f_\chi\neq 0$, we have $\langle \chi ,\alpha^\vee \rangle \leq 0$ for all $\alpha\in \Phi_{+}\setminus \Phi_{+,I'}$.
\end{assertionlist}
\end{lemma}

\begin{proof}
Consider the space $W$ of all $U'$-equivariant functions $h\colon G\to \AA^1$. Since $U'$ is normal in $P'$, it is clear that any right-translate of $h$ by an element of $P'$ is again $U'$-equivariant. Hence the space $W$ is a $P'$-representation. In particular, it decomposes with respect to the action of $T$. This shows (1). For the second assertion, by (1) we may assume $f=f_\chi$. Let $\phi_{\alpha}\colon \SL_2\to G$ denote the map attached to $\alpha$, as in \cite[9.2.2]{Springer-Linear-Algebraic-Groups-book}. It satisfies
\[
 \phi_\alpha 
 \left( \begin{pmatrix}
 1 & x \\ 0 & 1 
 \end{pmatrix}\right) = u_{\alpha}(x), \quad 
 \phi_\alpha 
 \left( \begin{pmatrix}
 1 & 0 \\ x & 1 
 \end{pmatrix}\right) = u_{-\alpha}(x).
\]
For a fixed element $g_0\in G$ and $\alpha\in \Phi_{+}\setminus \Phi_{+,I'}$, consider the map
\begin{equation}
f_\alpha \colon  \SL_2 \to \AA^1, \quad A\mapsto f\left(g_0 \phi_\alpha(A) \right).
\end{equation}
Let $V(m)\colonequals \Ind_{B_0}^{\SL_2}(\chi_m)$ where $B_0$ is the lower Borel subgroup of $\SL_2$ and $\chi_m$ is the character $\diag(x,x^{-1})\mapsto x^m$. It is immediate that $f_\alpha$ lies in the $\SL_2$-representation $V(-\langle \chi ,\alpha^\vee \rangle)$. We can clearly choose $g_0$ such that $f_\alpha$ is nonzero. In particular $V(-\langle \chi ,\alpha^\vee \rangle)\neq 0$ hence $\langle \chi ,\alpha^\vee \rangle\leq 0$.
\end{proof}

\begin{corollary}
Let $f\colon G\to \AA^1$ be a regular map satisfying $f(ugv)=f(g)$ for all $g\in G$ and all $(u,v)\in U\times V$. Then:
\begin{assertionlist}
\item We may decompose $f$ as 
\[f=\sum_{(\chi_1,\chi_2)} f_{\chi_1,\chi_2}\]
where $\chi_1,\chi_2\in X^*(T)$ and $f_{\chi_1,\chi_2}$ satisfies $f_{\chi_1,\chi_2}(t_1gt_2)=\chi_1(t_1)\chi_2(t_2)^{-1} f_{\chi_1,\chi_2}(g)$ for all $g\in G$, $t_1,t_2\in T$, as well as $f_{\chi_1,\chi_2}(ugv)=f_{\chi_1,\chi_2}(g)$ for all $g\in G$ and $(u,v)\in U\times V$.
\item For all $(\chi_1,\chi_2)$ such that $f_{\chi_1,\chi_2}\neq 0$, we have $\langle \chi_1 ,\alpha^\vee \rangle \geq 0$ for all $\alpha\in \Phi_{+}\setminus \Phi_{+,L}$ and $\langle \chi_2 ,\alpha^\vee \rangle \leq 0$ for all $\alpha\in \Phi_{+}\setminus \Phi_{+,M}$
\end{assertionlist}
\end{corollary}

\begin{proof}
The first assertion is proved as in Lemma \ref{lem-decomp}, noting that the space of $U\times V$-invariant regular functions is stable by the action of $P\times Q$. For the second assertion, apply the lemma to the functions $g\mapsto f(g)$ (resp. $g\mapsto f(w_0g^{-1}w_0)$) to obtain the inequality satisfied by $\chi_2$ (resp. $\chi_1$).
\end{proof}

\begin{corollary}\label{Sunip-cor}
The space $S_{\unip}$ decomposes as follows:
\begin{equation}
    S_{\unip}=\bigoplus_{(\chi_1,\chi_2)} S_{\unip}(\chi_1,\chi_2)
\end{equation}
where $S_{\unip}(\chi_1,\chi_2)$ is the subspace of functions $f\in S_{\unip}$ satisfying $f(t_1gt_2^{-1})=\chi_1(t_1)\chi_2(t_2)f(g)$. Furthermore, any $(\chi_1,\chi_2)$ such that $S_{\unip}(\chi_1,\chi_2)\neq 0$ satisfies $\langle \chi_1 ,\alpha^\vee \rangle \geq 0$ for all $\alpha\in \Phi_{+}\setminus \Phi_{+,L}$ and $\langle \chi_2 ,\alpha^\vee \rangle \leq 0$ for all $\alpha\in \Phi_{+}\setminus \Phi_{+,M}$.
\end{corollary}

We note that the map $\vartheta\colon S_{\unip}\to k[B_M]$ is not $T\times T$-equivariant. It maps $S_{\unip}(\chi_1,\chi_2)$ to the weight space  $k[B_M]_{\chi_2, w_{0,M}\chi_1}$.

\subsubsection{Unipotent-invariance cone}

We now start with a nonzero section $f\in H^0(\GZip^\mu,\Vcal_I(\lambda))$ for some $\lambda\in X_{+,I}^*(T)$. Our goal is to show that $\lambda$ satisfies certain constraints. First, view $f$ as an element of $S_{\unip}$. By Corollary \ref{Sunip-cor}, we may decompose $f$ as 
\begin{equation}
f=\sum_{\chi_1,\chi_2} f_{\chi_1,\chi_2}    
\end{equation}
where $f_{\chi_1,\chi_2}\in S_{\unip}(\chi_1,\chi_2)$. Furthermore, we have $\langle \chi_1 ,\alpha^\vee \rangle \geq 0$ for all $\alpha\in \Phi_{+}\setminus \Phi_{+,L}$ and $\langle \chi_2 ,\alpha^\vee \rangle \leq 0$ for all $\alpha\in \Phi_{+}\setminus \Phi_{+,M}$ whenever $(\chi_1,\chi_2)$ appears. Next, we apply $\vartheta\colon S_{\unip}\to k[B_M]$. By the discussion in section \ref{subsec-gp-prelim}, there exists $(\chi_1,\chi_2)$ such that $\vartheta(f_{\chi_1,\chi_2})$ is a nonzero element $h\in k[B_M]$. Recall also that the weight of $h$ with respect to the $T\times T$-action on $k[B_M]$ is $(\chi_2,w_{0,M}\chi_1)$. We can decompose $h$ as a sum of monomials of the form $h_{m,\xi}$ for $(m,\xi)\in \NN^{\Phi_{M,-}}\times X^*(T)$ as in section \ref{subsec-gp-prelim}. Since $\vartheta(f)\in S(\lambda)$, we have simultaneously:
\[
\begin{cases}
    \lambda &= \lambda(\chi_2,w_{0,M}\chi_1)= q \sigma^{-1}\chi_2+\chi_1\\
    \chi_1 & = -w_{0,M}\xi + \displaystyle{\sum_{\alpha\in \Phi_{M,-}} m_\alpha (w_{0,M}\alpha)} \\
\chi_2 & = \xi.
\end{cases}
\]
Note that in the above sum $w_{0,M}\alpha$ lies in $\Phi_{M,+}$. Therefore, putting everything together, we deduce that $\lambda$ satisfies the following condition: There exists a character $\chi_2\in X^*(T)$ such that
\[\begin{cases}
    & \lambda - q\sigma^{-1}(\chi_2) + w_{0,M}\chi_2 \ \textrm{is a sum of positive roots of $M$} \\
    &\langle \chi_2 ,\alpha^\vee \rangle \leq \min\left(0, \frac{1}{q}\langle \sigma(\lambda),\alpha^\vee \rangle \right) \quad \textrm{for all} \ \alpha\in \Phi_{+}\setminus \Phi_{+,M}.
\end{cases}\]

\begin{definition}
Let $C_{\unip}\subset X^*(T)$ be the set of $\lambda\in X^*(T)$ such that there exists a character $\chi_2\in X^*(T)$ satisfying the condition above. We call $C_{\unip}$ the unipotent-invariance cone.
\end{definition}
It is clear that $C_{\unip}$ is a saturated subcone of $X^*(T)$. We have shown that if $f\in H^0(\GZip^\mu,\Vcal_I(\lambda))$ is nonzero, then the weight of $f$ lies in $C_{\unip}$. Hence:
\begin{theorem}
We have $\langle C_{\zip} \rangle \subset C_{\unip}$.
\end{theorem}

The saturated cone $\langle C_{\unip} \rangle$ has a similar description as $C_{\unip}$, except that we allow a linear combination of positive roots of $M$ with non-negative rational coefficients.

\subsection{The split case}\label{sec-split}

We simplify the situation by making the following assumptions:
\begin{assertionlist}
\item $P$ is defined over $\FF_q$. In particular, we have $L=M$.
\item The group $G$ is split over $\FF_{q^2}$.
\end{assertionlist}
In particular, both conditions are satisfied if $G$ is split over $\FF_q$. For characters $\chi_2,\lambda \in X^*(T)$, write $\gamma=\lambda - q\sigma^{-1}(\chi_2) + w_{0,L}\chi_2$. We wish to express $\chi_2$ in terms of $\lambda$ and $\gamma$. Using the above assumptions, we find:
\begin{equation}
    \chi_2 = -\frac{1}{q^2-1} (w_{0,I} (\gamma-\lambda)+q\sigma(\gamma-\lambda)).
\end{equation}
For characters $\lambda_1$, $\lambda_2$, write $\lambda_1\leq_L \lambda_2$ if for all roots $\alpha\in \Phi_+\setminus \Phi_{L,+}$, we have $\langle \lambda_1-\lambda_2,\alpha^\vee \rangle \leq 0$. Under the assumptions (1)-(2), we deduce that any weight $\lambda\in C_{\unip}$ satisfies : There exists a character $\gamma\in X^*(T)$ which is a sum of positive roots of $L$ such that
\[\begin{cases}
    & w_{0,I}\lambda+q\sigma(\lambda) \leq_L w_{0,I}\gamma + q \sigma(\gamma) \\
    & w_{0,I}\lambda+\frac{1}{q}\sigma(\lambda)\leq_L w_{0,I}\gamma + q \sigma(\gamma) 
    \end{cases}\]
In particular, assume that $\alpha_1,\dots ,\alpha_m\in \Phi_+\setminus \Phi_{L,+}$ and that $\alpha_1^\vee+\dots +\alpha_m^\vee=\delta$ is a cocharacter in $X_*(L)$. Since $w_{0,I}\gamma + q \sigma(\gamma)$ is again a sum of roots of $L$, it is orthogonal to $\delta$. Let $\{1,\dots,m\}= S_1\sqcup S_2$ be any partition of $\{1,\dots,m\}$. We obtain
\begin{align*}
    \sum_{i\in S_1} \langle w_{0,I}\lambda + q\sigma(\lambda) ,\alpha^\vee_i\rangle + \sum_{i\in S_2} \langle w_{0,I}\lambda + \frac{1}{q}\sigma(\lambda),\alpha_i^\vee \rangle &\leq 0, \\
 \textrm{hence}\quad \langle \lambda,\delta\rangle + q \sum_{i\in S_1} \langle \sigma(\lambda),\alpha_i^\vee \rangle + \frac{1}{q} \sum_{i\in S_2} \langle \sigma(\lambda), \alpha_i^\vee\rangle &\leq 0
\end{align*}
If we assume that $\{\alpha_1,\dots,\alpha_m\}$ is stable by $\sigma$, then we can replace $\sigma(\lambda)$ by $\lambda$ in the above formula (using the partition $\sigma(S_1)\sqcup \sigma(S_2)$). In this case, we obtain
\begin{align*}
 (q+1) \sum_{i\in S_1} \langle \lambda,\alpha_i^\vee \rangle + \left(\frac{1}{q}+1 \right) \sum_{i\in S_2} \langle \lambda, \alpha_i^\vee\rangle &\leq 0\\
 \sum_{i\in S_1} \langle \lambda,\alpha_i^\vee \rangle + \frac{1}{q} \sum_{i\in S_2} \langle \lambda, \alpha_i^\vee \rangle &\leq 0\\
\end{align*}
where we divided the equation by $q+1$. Consider the action of $W_L\rtimes \Gal(\FF_{q^2}/\FF_q)$ on $\Phi$. Note that this action preserves positivity of roots outside of $L$. In particular, the set $\Phi_+\setminus \Phi_{L,+}$ is stable under $W_L\rtimes \Gal(\FF_{q^2}/\FF_q)$. It is not always the case that $\Phi_+\setminus \Phi_{L,+}$ consists of a single orbit. Let $\Ocal\subset \Phi_+\setminus \Phi_{L,+}$ be an orbit. Define $\delta_\Ocal$ as the sum of all coroots in $\Ocal$: 
\[\delta_{\Ocal}\colonequals \sum_{\alpha\in \Ocal} \alpha^\vee.\]
For any root $\beta\in \Delta_L$, the reflection $s_\beta$ satisfies $s_\beta(\delta_\Ocal)=\delta_\Ocal$, hence $\langle \beta, \delta_\Ocal\rangle=0$. It follows that $\delta_\Ocal\in X_*(L)$. Moreover, it is clear that $\sigma(\delta_{\Ocal})=\delta_\Ocal$. The above discussion applies to $\delta_\Ocal$ and shows that $C_{\unip}$ satisfies all the inequalities of the type
\begin{equation}\label{ineq-subset}
\Gamma_{\Ocal,S}(\lambda)\colonequals \sum_{\alpha \in \Ocal\setminus S} \langle \lambda,\alpha^\vee \rangle \ + \ \frac{1}{q} \sum_{\substack{\alpha\in S}} \langle  \lambda, \alpha^\vee \rangle \  \leq \ 0
\end{equation}
for any subset $S\subset \Ocal$. Denote by $C_\Ocal\subset X^*(T)$ the cone of $\lambda$ satisfying the inequalities \eqref{ineq-subset} for all subset $S\subset \Ocal$. Note that we could have defined $\delta_{\Ocal}$ similarly when $\Ocal$ is a union of $W_L\rtimes \Gal(\FF_{q^2}/\FF_q)$-orbits. In particular, we may speak of the cone $C_{\Phi_+\setminus \Phi_{L,+}}$. However, note that if $\Ocal=\Ocal_1\sqcup \Ocal_2$ and $S\subset \Ocal$ is any subset, we have
\begin{equation}
    \Gamma_{\Ocal,S}(\lambda) = \Gamma_{\Ocal_1,S\cap \Ocal_1}(\lambda)+\Gamma_{\Ocal_2,S\cap \Ocal_1}(\lambda).
\end{equation}
Hence we deduce that $C_{\Ocal_1}\cap C_{\Ocal_2}\subset C_\Ocal$ and thus we can reduce to considering the cones $C_{\Ocal}$ when $\Ocal$ is an $W_L\rtimes \Gal(\FF_{q^2}/\FF_q)$-orbit in $\Phi_+\setminus \Phi_{L,+}$. We define the orbit cone $C_{\rm orb}$ as follows:
\begin{equation}
    C_{\rm orb} \colonequals \bigcap_{\substack{ \textrm{orbits} \\ \Ocal \subset \Phi_+\setminus \Phi_{L,+}}} C_{\Ocal}
\end{equation}
where the intersection is taken over all $W_L\rtimes \Gal(\FF_{q^2}/\FF_q)$-orbits $\Ocal \subset \Phi_+\setminus \Phi_{L,+}$. By the above discussion, we have inclusions
\begin{equation}\label{contain-zip-Gamma}
    C_{\zip} \ \subset \ C_{\unip} \ \subset \ C_{\rm orb} \ \subset \ C_{\Phi_+\setminus \Phi_{L,+}}.
\end{equation}
All inclusions above are in general strict. We illustrate the difference between $C_{\zip}$, $C_{\rm orb}$ and $C_{\unip}$ in section \ref{subsec-GS-unitary-Sp} in the case $G=\Sp(6)_{\FF_q}$. We were not able to determine $C_{\unip}$ in general (or even under assumptions (1)-(2)), but it will be sufficient for our purposes to work with the cone $C_{\rm orb}$ since it already provides a sharp approximation. In certain cases, the inclusion $C_{\zip}\subset C_{\Phi_+\setminus \Phi_{L,+}}$ will be enough for our purpose, as in the proof of Theorem \ref{zip-asymp} in section \ref{sec-asymp-zip}. However, in the case $G=\Sp_{2n,\FF_q}$, the set $\Phi_+\setminus \Phi_{L,+}$ contains two orbits, and the cone $C_{\Phi_+\setminus \Phi_{L,+}}$ is strictly coarser than $C_{\Ocal}$, where $\Ocal$ is the orbit of the unique simple root outside of $L$. When we want to emphasize the dependance of $C_{\Ocal}$ on the prime power $q$, it will be convenient to write $C_{\Ocal,q}$. Similarly, we write $\Gamma_{\Ocal,S,q}$ for the function $\Gamma_{\Ocal,S}$.

The number of inequalities defining the cone $C_{\Ocal}$ is the cardinality of the powerset of $\Ocal$, which can be quite large. However, we are eventually interested in the cone $C_{\zip}$, which is contained in the $L$-dominant cone $X^*_{+,I}(T)$. Therefore, it is sufficient to consider the intersection $C_{\Ocal}\cap X^*_{+,I}(T)$. Looking at concrete examples, we see that this intersection is cut out in $X^*_{+,I}(T)$ by inequalities $\Gamma_{\Ocal,S}(\lambda)\leq 0$ for a rather small number of subsets $S\subset \Ocal$ (the other subsets do not contribute to this intersection). The following notion seems to be relevant:
\begin{definition}
A subset $S\subset \Phi_+\setminus \Phi_{L,+}$ is $L$-minimal if it satisfies the following condition: For any $\alpha\in S$ and any $\beta\in \Delta_L$ such that $\alpha-\beta\in \Phi_+$, we have $\alpha-\beta\in S$. Denote by $\Min(\Phi_+\setminus \Phi_{L,+})$ the set of all $L$-minimal subsets of $\Phi_+\setminus \Phi_{L,+}$.
\end{definition}
For $w\in W^I$, define a subset
\begin{equation}
    \Min(w) \colonequals \left\{ \alpha\in \Phi_+\setminus \Phi_{L,+} \relmiddle| \ \ell(w s_{\alpha}) < \ell(w) \right\}.
\end{equation}
Then one can show that $\Min(w)$ is a $L$-minimal subset, and the map $w\mapsto \Min(w)$ induces a bijection $W^I\to \Min(\Phi_+\setminus \Phi_{L,+})$. For an orbit $\Ocal \subset \Phi_+\setminus \Phi_{L,+}$, define $\Min(\Ocal)$ as the set of $L$-minimal subsets contained in $\Ocal$, i.e:
\begin{equation}
    \Min(\Ocal)=\Min(\Phi_+\setminus \Phi_{L,+})\cap \Pcal(\Ocal).
\end{equation}
Then, we expect the following to hold:
\begin{equation}
C_{\Ocal}\cap X^*_{+,I}(T) = \{ \lambda\in X^*_{+,I}(T) \ \mid \ \Gamma_{\Ocal,S}(\lambda)\leq 0 \ \textrm{for all} \ S\in \Min(\Ocal) \}.   
\end{equation}
In particular, only a small number of subsets $S$ contribute nontrivially. The above can be easily checked this in the cases $G=\GL_{n,\FF_q}$ and $G=\Sp(2n)_{\FF_q}$ considered in sections \ref{subsec-GS-unitary-Sp} and \ref{sec-Shim-van}, but we have not proved it in general. It will be convenient to define the following set, which we call the $L$-minimal cone:
\begin{equation}\label{L-min-cone}
C_{L-\Min} = \{ \lambda\in X^*(T) \ \mid \ \Gamma_{\Ocal,S}(\lambda)\leq 0 \ \textrm{for all orbits $\Ocal$ and all} \ S\in \Min(\Ocal) \}.   
\end{equation}
Hence, at least in the cases considered in sections \ref{subsec-GS-unitary-Sp} and \ref{sec-Shim-van}, we have $C_{\rm orb}\cap X^*_{+,I}(T) = C_{L-\Min}\cap X^*_{+,I}(T)$

\subsection{Asymptotic zip cone}\label{sec-asymp-zip}

We apply the results of the previous section to study the asymptotic behaviour of the cone $C_{\zip,p}$ in families. We may work in a more general setting, independently of the theory of Shimura varieties. We let $\mathbf{G}$ be a reductive $\QQ$-group endowed with a cocharacter $\mu\colon \GG_{\mathrm{m}}\to \mathbf{G}_{\overline{\QQ}}$. There exists an integer $N\geq 1$ such that $\mathbf{G}$ admits a reductive $\ZZ[\frac{1}{N}]$-model $\Gcal$. Furthermore, there is a number field $E$ such that $\mu$ extends to a cocharacter of $\Gcal_R$ where $R=\Ocal_E[\frac{1}{N}]$. From this, we obtain a zip datum $(G_p,\mu_p)$ for all primes $p$ not dividing $N$. We choose a Borel pair $(\mathbf{B},\mathbf{T})$ of $\mathbf{G}$ and we may assume that it has a model $(\Bcal,\Tcal)$ in $\Gcal$. We obtain compatible Borel pairs $(B_p,T_p)$ for all $G_p$, and we may identify their character groups and their root data. Write $C_{\zip,p}$ for the zip cone of the zip datum attached to $(G_p,\mu_p)$. We may view all the cones $C_{\zip,p}$ inside the same character group $X^*(\mathbf{T})$. We have the following:

\begin{theorem}\label{zip-asymp}
The family $(\langle C_{\zip,p} \rangle)_{p}$ is asymptotic to $C_{\GS}$.
\end{theorem}

\begin{proof}
Since $\mu$ is defined over the number field $\mathbf{E}$, the cocharacter $\mu_p$ will be defined over $\FF_p$ for any $p$ which is split in $\mathbf{E}$. Similarly, choose a number field $\mathbf{F}$ such that $\mathbf{G}_{\mathbf{F}}$ is split. Then for any prime $p$ split in $\mathbf{F}$, the group $G_p$ is split over $\FF_p$. In particular, there are infinitely many such primes. It suffices to show for all $N\geq 1$:
\begin{equation}
    \bigcap_{\substack{G_p \textrm{ split}\\ p\geq N}} \langle C_{\zip,p} \rangle = C_{\GS}.
\end{equation}
For such primes, we may apply the results of the previous section. Define the cone $C_{\Phi_+\setminus \Phi_{L,+},p}$ as in section \ref{sec-split} (where we take $q=p$). By the inclusions \eqref{contain-zip-Gamma}, it suffices to show that the intersection of all the cones $C_{\Phi_+\setminus \Phi_{L,+},p}$ (for $p$ such that $G_p$ is split) coincides with $C_{\GS}$. Let $\lambda$ be a character in all the $C_{\Phi_+\setminus \Phi_{L,+},p}$. For $\beta\in \Phi_+\setminus \Phi_{L,+}$, consider the subset $S= \Phi_+\setminus (\Phi_{L,+}\cup \{\beta\})$. By assumption, $\lambda$ satisfies
\begin{equation}
   \Gamma_{\Phi_+\setminus \Phi_{L,+},S,p}(\lambda)=\langle \lambda, \beta^\vee \rangle + \frac{1}{p} \sum_{\substack{\alpha \in \Phi_+\setminus \Phi_{L,+} \\ \alpha\neq \beta }} \langle \lambda,\alpha^\vee\rangle \leq 0. 
\end{equation}
Passing to the limit on $p$, we obtain $\langle \lambda, \beta^\vee \rangle\leq 0$ for any $\beta\in \Phi_+\setminus \Phi_{L,+}$. Since $C_{\zip}\subset X^*_{+,I}(T)$ and $C_{\GS}\subset \langle C_{\zip,p}\rangle$ for all $p$, we deduce that the intersection of the cones $\langle C_{\zip,p}\rangle$ coincides with $C_{\GS}$.
\end{proof}

The proof shows that the family $(C_{\Phi_+\setminus \Phi_{L,+},p}\cap X_{+,I}^*(T))_p$ is a GS-approximation of the family $(C_{\zip,p})_p$ (this is a slight abuse of terminology, since we have not defined $C_{\Phi_+\setminus \Phi_{L,+},p}$ for general $p$). Theorem \ref{zip-asymp} combined with Conjecture \ref{conj-Sk} indicates that we should expect a similar result for the Shimura cone family $(\langle C_{K,p} \rangle)_p$ (recall that $C_{K,p}\colonequals C_K(\overline{\FF}_p)$). In particular, we expect that $C_{K,p}$ is contained in $C_{\Phi_+\setminus \Phi_{L,+},p}$ for all $p$ where $G_p$ is split.

\subsection{GS-approximations for $\GL_{n}$ and $\Sp(2n)$}\label{subsec-GS-unitary-Sp}
We give explicit equations for $C_{\Ocal}$ and $C_{\rm orb}$ in the case of general linear groups and symplectic groups.

\subsubsection{General linear groups}\label{GLn-sec}

Set $G=\GL_{n,\FF_q}$ (as usual, we take $q=p$ in the context of Shimura varieties). Consider the cocharacter $\mu \colon \GG_{\mathrm{m},k}\to G_k$ by $\mu(x)=\diag(xI_r,I_s)$ with $r+s=n$. Write $\Zcal_\mu=(G,P,L,Q,M,\varphi)$ for the attached zip datum. If $(u_1,\dots ,u_n)$ denotes the canonical basis of $k^n$, then $P$ is the stabilizer of $V_P\colonequals \Span_k(u_{r+1},\dots , u_n)$ and $Q$ is the stabilizer of $V_Q\colonequals \Span_k(u_{1},\dots , u_r)$. Let $B$ denote the lower-triangular Borel and $T$ the diagonal torus. The Levi subgroup $L=P\cap Q$ is isomorphic to $\GL_{r,\FF_q}\times \GL_{s,\FF_q}$. Identify $X^*(T)=\ZZ^n$ such that $(a_1,\dots ,a_n)\in \ZZ^n$ corresponds to the character $\diag(x_1,\dots ,x_n)\mapsto \prod_{i=1}^n x_i^{a_i}$. The simple roots with respect to $B$ are $\{\alpha_i\}_{1\leq i \leq n-1}$ where
\begin{equation}
    \alpha_i= e_i-e_{i+1}
\end{equation}
and $(e_i)_{1\leq i \leq n}$ denotes the canonical basis of $\ZZ^n$.
For general $(r,s)$, we do not know a description of $C_{\zip}$ or even $\langle C_{\zip} \rangle$. The cones $X_{+,I}^*(T)$ and $C_{\GS}$ are given by
\begin{align*}
     X_{+,I}^*(T) &= \{(a_1,\dots,a_n)\in \ZZ^n \mid a_1\geq \dots \geq a_r \ \textrm{and} \ a_{r+1}\geq \dots \geq a_n \} \\
     C_{\GS} &= \{(a_1,\dots,a_n)\in X_{+,I}^*(T) \mid a_1\leq a_n \}.
\end{align*}
In this case, the group $W_L$ acts transitively on $\Phi_+\setminus \Phi_{L,+}$.

First, we explicit the set $\Min(\Phi^+\setminus \Phi^+_L)$. This set is in bijection with the set of finite decreasing sequences $x=(x_j)_{1\leq j \leq s}$ such that $r\geq x_1\geq x_2\geq\dots \geq x_s \geq 0$. To each such sequence, we can attach the $L$-minimal subset
\begin{equation}
   S_x\colonequals \{ e_i-e_j \ \mid \ r+1-x_j \leq i \leq r, \ r+1\leq j \leq n \}.
\end{equation}
Write simply $\Gamma_x(\lambda)$ for the function $\Gamma_{\Phi^+\setminus \Phi^+_L,S_x}(\lambda)$. If we write $\lambda=(a_1,\dots,a_n)$, we have:
\begin{equation}
    \Gamma_x(\lambda)=\sum_{j=r+1}^n \left( \sum_{i=1}^{r-x_{j-r}} (a_i-a_j) + \frac{1}{q}\sum_{i=r-x_{j-r}+1}^r (a_i-a_j)\right).
\end{equation}
Hence, the $L$-minimal cone (\eqref{L-min-cone}) is given as follows:
\begin{equation*}
    C_{L-\Min} = \{ \lambda\in X^*(T) \ \mid \ \Gamma_x(\lambda)\leq 0, \ \textrm{for all decreasing sequences $x$} \}.
\end{equation*}
On this example, the expected equality $C_{L-\Min}\cap X_{+,I}^*(T) = C_{\rm orb} \cap X_{+,I}^*(T)$ (see end of section \ref{sec-split}) is a straightforward computation. When $(r,s)=(n-1,1)$, we obtain that $C_{L-\Min}$ is given by the following equations
\begin{equation}\label{eq-n1}
\sum_{i=1}^k (a_i-a_n) + \frac{1}{q} \sum_{i=k+1}^{n-1} (a_i-a_n)\leq 0 \quad \textrm{for all } 0\leq k \leq n-1.
\end{equation}
Furthermore, one can check that the intersection $C_{L-\Min}\cap X_{+,I}^*(T)$ coincides with the $\lambda=(a_1,\dots, a_n)\in X_{+,I}^*(T)$ satisfying the inequalities \eqref{eq-n1} for $k=1,\dots, n-1$ (the inequality for $k=0$ can be omitted).

Consider the case $(r,s)=(2,2)$. The $L$-dominant cone $X_{+,I}^*(T)$ is the set of $\lambda=(a_1,a_2,a_3,a_4)\in \ZZ^4$ such that $a_1\geq a_2$ and $a_3\geq a_4$. The set ${}^I W$ has cardinality $\frac{4!}{2! 2!}=6$. When intersecting with $X_{+,I}^*(T)$, three of the corresponding 6 equations become redundant. Specifically, $C_{\rm orb}\cap X_{+,I}^*(T)$ is the set of $\lambda=(a_1,a_2,a_3,a_4)\in X_{+,I}^*(T)$ satisfying
\begin{equation}
  \begin{cases}
     2q a_1+ 2 a_2 - (q+1) a_3 - (q+1) a_4 \leq 0 \\
(q+1)a_1+2 a_2 - 2 a_3 - (q+1)a_4 \leq 0 \\
(q+1)a_1+ (q+1)a_2- 2 a_3 - 2q a_4\leq 0    \end{cases}
\end{equation}      
In the case of a unitary group of signature $(2,2)$ at a split prime of good reduction, Conjecture \ref{conj-Sk} holds by \cite[Theorem 4.2.8]{Goldring-Koskivirta-divisibility}. Furthermore, this case is of Hasse-type (\cite[Definition 5.1.6]{Goldring-Imai-Koskivirta-weights}), hence we have $\langle C_{\zip} \rangle=\langle C_{\Hasse}\rangle$ by \loccitn, Theorem 5.3.1. Therefore, if $X$ denotes any $\overline{\FF}_q$-scheme satisfying Assumption \ref{assume-atp} (for example, the corresponding unitary Shimura variety), we have:
\begin{equation}
    \langle C_{X} \rangle = \langle C_{\zip} \rangle=\langle C_{\Hasse}\rangle = \{(a_1,a_1,a_3,a_4)\in X_{+,I}^*(T) \mid q(a_1-a_4)+(a_2-a_3) \leq 0\}. 
\end{equation}
We see on this example that the actual cones $\langle C_X \rangle$ and $\langle C_{\zip} \rangle$ have a much simpler expression than the approximation $C_{\rm orb}$. However, for general groups we do not have an expression for either $\langle C_X \rangle$ or $\langle C_{\zip} \rangle$. Even worse, we could not prove that they are polyhedral cones.

\subsubsection{Sympectic groups}
We first give some notations for an arbitrary symplectic group. Let $(V_0,\psi)$ be a non-degenerate symplectic space over $\FF_q$ of dimension $2n$, for some integer $n\geq 1$. After choosing an appropriate basis $\Bcal$ for $V_0$, we assume that $\psi$ is given by the matrix
\begin{equation}
\begin{pmatrix}
& -J \\
J&\end{pmatrix}
\quad \textrm{ where } \quad
J\colonequals \begin{psmallmatrix}
&&1 \\
&\iddots& \\
1&&
\end{psmallmatrix}.
\end{equation} 
Define $G$ as follows:
\begin{equation}\label{group}
G(R) = \{f\in \GL_{\FF_{q}}(V_0\otimes_{\FF} R) \mid  \psi_R(f(x),f(y))=\psi_R(x,y), \ \forall x,y\in V_0\otimes_{\FF_q} R \}
\end{equation}
for all $\FF_q$-algebras $R$. Identify $V_0=\FF_q^{2n}$ via $\Bcal$ and view $G$ as a subgroup of $\GL_{2n,\FF_q}$. Fix the $\FF_q$-split maximal torus $T$ given by diagonal matrices in $G$, i.e.
\begin{equation}
T(R)\colonequals \{ \diag_{2n}(x_1,\ldots ,x_n,x^{-1}_n,\ldots,x^{-1}_1) \mid x_1, \ldots ,x_n\in R^\times \}.
\end{equation}
Define $B$ as the Borel subgroup of $G$ consisting of the lower-triangular matrices in $G$. For a tuple $(a_1,\dots,a_n)\in \ZZ^n$, define a character of $T$ by mapping $\diag_{2n}(x_1, \dots ,x_n,x^{-1}_n, \dots ,x^{-1}_1)$ to $x_1^{a_1} \cdots x_n^{a_n}$. From this, we obtain an identification $X^*(T) = \ZZ^n$. Denoting by $(e_1,\dots ,e_n)$ the standard basis of $\ZZ^n$, the $T$-roots of $G$ and the $B$-positive roots are respectively
\begin{align}
\Phi&\colonequals \{\pm e_i \pm e_j \mid 1\leq i \neq j \leq n\} \cup \{ \pm 2e_i \mid 1\leq i \leq n \}, \\
\Phi_+&\colonequals \{e_i \pm e_j \mid 1\leq i< j \leq n\} \cup \{ 2e_i \mid 1\leq i \leq n\} 
\end{align}
and the $B$-simple roots are $\Delta\colonequals \{\alpha_1,\dots , \alpha_{n-1},\beta\}$ where 
\begin{align*}
\alpha_i&\colonequals e_{i}-e_{i+1} \textrm{ for } i=1,...,n-1 ,\\ \beta&\colonequals 2e_n.
\end{align*}
The Weyl group $W\colonequals W(G,T)$ can be identified with the group of permutations $\sigma \in \Sfr_{2n}$ satisfying $\sigma(i)+\sigma(2n+1-i)=2n+1$ for all $1\leq i \leq 2n$. 
Define a cocharacter $\mu \colon \GG_{\mathrm{m},\FF_q}\to G$ by $z\mapsto \diag(zI_n,z^{-1}I_n)$. Write $\Zcal\colonequals (G,P,L,Q,M,\varphi)$ for the associated zip datum (since $\mu$ is defined over $\FF_q$, we have $M=L$). Concretely, if we denote by $(u_i)_{i=1}^{2n}$ the canonical basis of $k^{2n}$, then $P$ is the stabilizer of $V_{0,P}=\Span_k(u_{n+1},...,u_{2n})$ and $Q$ is the stabilizer of $V_{0,Q}=\Span_k(u_{1},...,u_{n})$. The intersection $L\colonequals P\cap Q$ is a common Levi subgroup and there is an isomorphism $\GL_{n,\FF_q}\to L$, $A\mapsto \theta(A)$, where:
\begin{equation}
\theta(A):=\left(\begin{matrix}\label{deltadef}
A& \\ & J {}^t\! A^{-1} J
\end{matrix} \right).
\end{equation}
Under the identification $X^*(T)=\ZZ^n$, we have:
\begin{align*}
    X^*_{+,I}(T) & = \{ (a_1,\dots ,a_n)\in \ZZ^n \mid \  a_1\geq \dots \geq a_n\} \\
    C_{\GS} &= \{ (a_1,\dots ,a_n)\in X^*_{+,I}(T) \mid \ a_1\leq 0 \}
\end{align*}
We do not know the general form of the cone $C_{\zip}$ outside the case $n=2$ (see \cite{Koskivirta-automforms-GZip}). For $n=3$, we determined $\langle C_{\zip} \rangle$ in \loccitn. For $n\geq 4$, neither $C_{\zip}$ nor its saturation $\langle C_{\zip} \rangle$ are known. Some approximations by subcones (Hasse cone, highest weight cone) were constructed in \loccitn. These notions were generalized to arbitrary groups in \cite{Goldring-Imai-Koskivirta-weights}.

Next, we explicit the results of the previous section and give an upper bound on $C_{\zip}$. There are two $W_L$-orbits in $\Phi_+\setminus \Phi_{L,+}$, given by
\begin{align*}
    &\Ocal_1=\{ e_i \ \mid \ 1\leq i \leq n \} \\
    &\Ocal_2=\{ e_i+e_j \ \mid \ 1\leq i< j \leq n \}.
\end{align*}
It turns out that the cone $C_{\Ocal_2}$ is coarser that $C_{\Ocal_1}$, so we will only consider $C_{\Ocal_1}$. One can prove that the cone $C_{\Ocal_1}\cap X_{+,I}^*(T)$ is the set of $\lambda=(a_1,\dots,a_n)\in X_{+,I}^*(T)$ satisfying
\begin{equation}
\sum_{i=1}^k a_i + \frac{1}{q} \sum_{i=k+1}^n a_i\leq 0 \quad \textrm{for all } 1\leq k \leq n-1.
\end{equation}
Therefore, the cone $C_{\rm orb}\cap X_{+,I}^*(T)$ is also given by the above inequalities. Note the similarities between the cases $G=\Sp(2n)$ and $G=\GL_{n+1}$ of signature $(n,1)$. Namely, if we set $a_{n+1}=0$ in the latter, we recover the equations for $\Sp(2n)$. As explained in \cite[\S 4.2.2]{Goldring-Koskivirta-divisibility}, there is a correspondence between automorphic forms on the corresponding stacks of $G$-zips for these two groups. Even though the number of $W_L$-orbits are different for these groups, this correspondence persists for the approximation cones $C_{\rm orb}$.

In the graph below, we consider the case $G=\Sp(6)$. We illustrate the approximations $C_{\rm orb}$, $C_{\unip}$ of the cone $\langle C_{\zip}\rangle$. Note that $X^*(T)=\ZZ^3$ is 3-dimensional, so to simplify we represent a slice of the cones. Hence, each dot on the picture represents a half-line from the origin. For a cone $C\subset X^*(T)$, write $C^{+,I}$ for its intersection with $X_{+,I}^*(T)$. We have:
\begin{align}
    C^{+,I}_{\rm orb} =  C^{+,I}_{L-\Min} &= \{ (a_1,a_2,a_3)\in X_{+,I}^*(T) \ \mid \ a_1+\frac{1}{q}(a_2+a_3) \leq 0, \ a_1+a_2+\frac{1}{q}a_3 \leq 0 \} \\
    C^{+,I}_{\unip} &= \{ (a_1,a_2,a_3)\in X_{+,I}^*(T) \ \mid \ a_1+\frac{1}{q}(a_2+a_3) \leq 0, \ q a_1+q^2a_2+a_3 \leq 0 \} \\
 C_{\zip} &= \{ (a_1,a_2,a_3)\in X_{+,I}^*(T) \ \mid \ q^2a_1+a_2+qa_3 \leq 0, \ q a_1+q^2a_2+a_3 \leq 0 \} 
\end{align}
As one sees on the figure below, the inclusions $C_{\zip}\subset C^{+,I}_{\unip} \subset C^{+,I}_{\rm orb}$ are strict.

\begin{figure}[H]
    \centering
    \includegraphics[width=16cm]{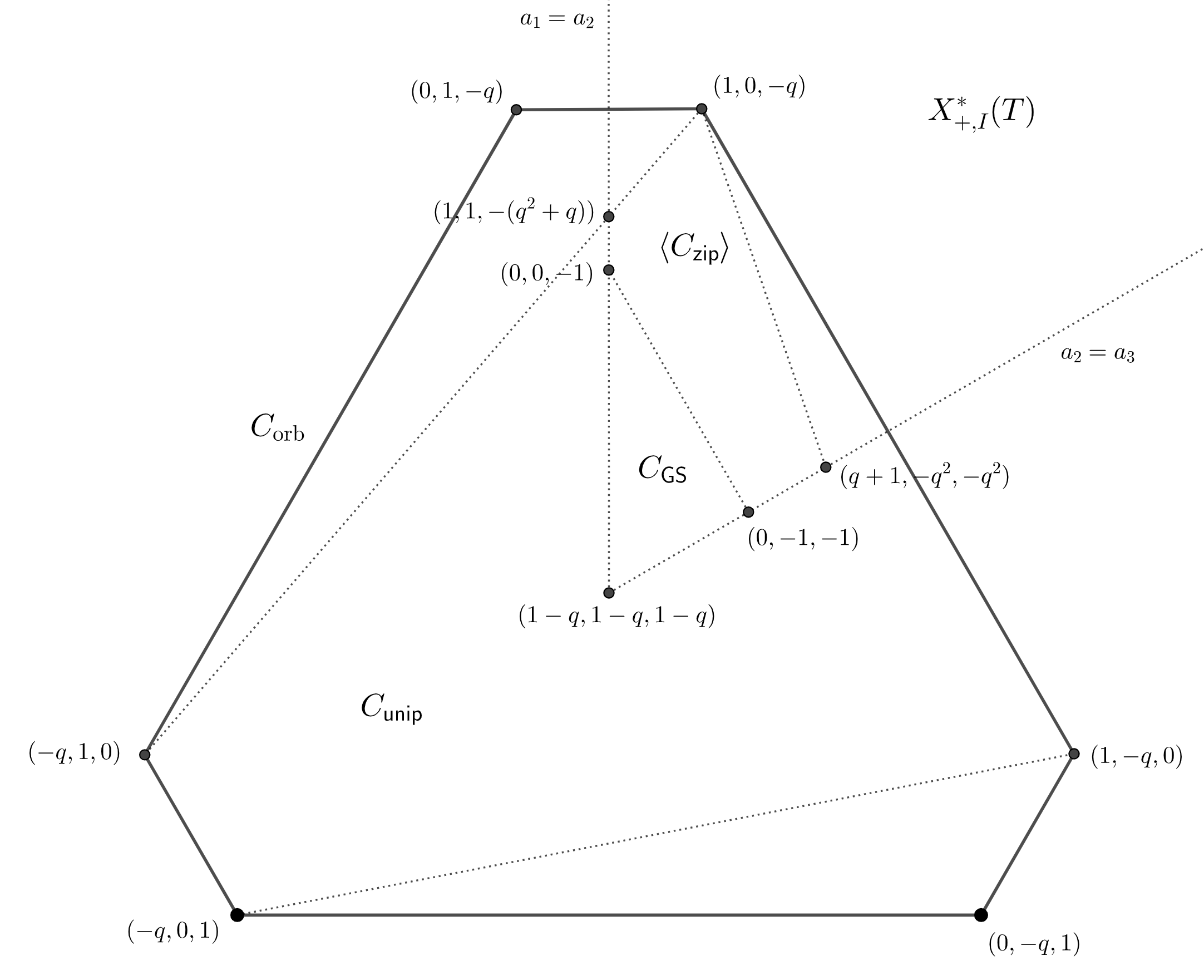}
    \caption{The case $\Sp(6)_{\FF_q}$}
    \label{fig:Sp6}
\end{figure}

\section{Vanishing at a fixed prime for unitary Shimura varieties}\label{sec-Shim-van}

In this section, we take $G=\GL_{n,\FF_q}$ and consider the setting of section \ref{GLn-sec}. In particular, $\mu\colon \GG_{\textrm{m},k}\to G_k$ is the cocharacter $x\mapsto \diag(xI_r,I_s)$ for $r+s=n$ and $\Zcal_\mu=(G,P,L,Q,M,\varphi)$ is the attached zip datum. In section \ref{sec-n1}, we will specialize to the case $(r,s)=(n-1,1)$.

\subsection{Partial Hasse invariants}

We let $S_n=W$ be the group of permutations of $\{1,\dots,n\}$. We start by recalling the following criterion for determining the set $E_w$ for $w\in S_n$ (see \eqref{Ew-def}). For $1\leq i\neq j \leq n$, we denote by $(i \ j)\in S_n$ the transposition exchanging $i$ and $j$.

\begin{proposition}\label{prop-low-neigh}
Let $1\leq i < j \leq n$. Then $w \times (i \ j)$ is a lower neighbour of $w$ if and only if the folllowing conditions hold
\begin{assertionlist}
\item $\sigma(i)> \sigma(j)$,
\item There is no $i<k<j$ such that $\sigma(j)<\sigma(k)<\sigma(i)$.
\end{assertionlist}
\end{proposition}

We may represent this criterion visually as follows: Consider the submatrix of $w$ whose corners are $(i,\sigma(i))$ and $(j,\sigma(j))$. Condition (i) says that $(i,\sigma(i))$ is the lower left corner of this matrix, and $(j,\sigma(j))$ is the upper right corner. Condition (ii) says that all the coefficients of this submatrix are zero except for these two corners.

\[ \begin{tikzpicture}[baseline=(math-axis),every right delimiter/.style={xshift=-3pt},every left delimiter/.style={xshift=3pt}]%
\matrix [matrix of math nodes,left delimiter=(,right delimiter=)] (matrix) 
{
|(m11)| & |(m12)|  & |(m13)| & |(m14)| & |(m15)| & |(m16)|  \\
|(m21)| & |(m22)| 0  & |(m23)| & |(m24)| & |(m25)| 1 & |(m26)|  \\
|(m31)| & |(m32)|  & |(m33)| & |(m34)| & |(m35)| & |(m36)|  \\
|(m41)| & |(m42)|  & |(m43)| & |(m44)| & |(m45)| & |(m46)|  \\
|(m51)| & |(m52)| 1  & |(m53)| & |(m54)| & |(m55)| 0 & |(m56)|  \\
|(m61)| & |(m62)|  & |(m63)| & |(m64)| & |(m65)| & |(m66)| \\ 
};
\draw[dashed] (m22.north west) -- (m25.north east) -- (m55.south east) -- (m52.south west) -- cycle;
\coordinate (math-axis) at ($(matrix.center)+(0em,-0.25em)$);
\end{tikzpicture}\]

\begin{definition}\label{syst-Hasse}
We say that $w\in W$ admits a system of partial Hasse invariants if the elements $\alpha^\vee$ for $\alpha\in E_w$ are linearly independent in $X_*(T)_\QQ$.
\end{definition}

If $w$ admits a system of partial Hasse invariants, then for each $\alpha\in E_w$, we can find $\chi\in X^*(T)$ satisfying Conditions (a) and (b) of Definition \ref{sep-syst-E}. This will be used to construct a separating system in section \ref{main-sec-n1}. Let us introduce some non-standard terminology: Let $w\in S_n$ be a permutation. A triplet $(i,j,k)$ satisfying $i<j<k$ and $w(i)<w(j)$ and $w(k)<w(j)$ will be called a V-shape. Furthermore, if $w(i)<w(k)$, we call it a $\sqrt{}$-shape.

\begin{lemma}\label{V-shape}
Assume that $w$ has no $\sqrt{}$-shape. Then $w$ admits a system of partial Hasse invariants.
\end{lemma}

\begin{proof}
For a transposition $t=(i \ j)$, put $t_-\colonequals \min\{i,j\}$ and $t_+\colonequals \max\{i,j\}$. Since $w$ has no $\sqrt{}$-shape, it is clear that the map $E_w \to \{1,\dots, n\}$, $t \mapsto t_+$ is injective. This implies that the elements $(\alpha^\vee)_{\alpha\in E_w}$ are linearly independent.
\end{proof}

\subsection{Auxilliary sequence}\label{subsec-auxil}

For $1\leq d\leq n$, define the matrix
\begin{equation}
    \Lambda_d \colonequals 
   \left( \begin{array}{ccc|ccc}
         & & &1 & &  \\
         & & & &\ddots &  \\
         & & & & & 1 \\ \hline
         & & 1 & & &  \\
         & \iddots& & & &  \\
         1& & & & &  
    \end{array}\right)
\end{equation}
where the upper right block has size $d\times d$ and the lower left block has size $(n-d)\times (n-d)$. For example, $\Lambda_1=w_0$ is the longest element of $W=S_n$ and $\Lambda_n=I_n$ is the identity element.

For two elements $w,w'$ such that $w>w'$, we define a path from $w$ to $w'$ to be a sequence $w_1,\dots,w_N$ satisfying the following conditions:
\begin{definitionlist}
\item $w_1=w$ and $w_N=w'$.
\item $w_1>\dots > w_N$ and $\ell(w_{i+1})=\ell(w_{i})-1$ for each $i=1,\dots,N-1$.
\end{definitionlist}
For $1\leq d <n$, we construct a path from $\Lambda_d$ to $\Lambda_{d+1}$ as follows: We multiply $\Lambda_d$ successively on the right by the transpositions $(n-d \ \ n-i+1)$ for $i=1,\dots, d$. In other words, we define $w^{(d)}_1=\Lambda_d$ and for $2\leq i \leq d+1$,
\begin{equation}
    w^{(d)}_i=\Lambda_d (n-d \ \ n) (n-d \ \ n-1)\dots (n-d \ \ n-i+2).
\end{equation}
Then $(w^{(d)}_1,\dots , w^{(d)}_{d+1})$ is a path from $\Lambda_d$ to $\Lambda_{d+1}$. At each step, the coefficient on the $n-d$-th column of the matrix moves up by one. Moreover, the last $d$ coefficients are in increasing order at each step of the sequence.

\begin{lemma}
Each element in the sequence $(w^{(d)}_1,\dots , w^{(d)}_{d+1})$ admits a system of partial Hasse invariants.
\end{lemma}

\begin{proof}
Every element in the sequence has no V-shape (in particular no $\sqrt{}$-shape), hence the result follows from Lemma \ref{V-shape}.
\end{proof}

The number of lower neighbours of $w^{(d)}_{i}$ is exactly $n-1$ for all $1\leq d<n-1$. Furthermore, for $1\leq d<n-1$, the set $E_{w^{(d)}_{i}}$ can be partitioned into three subsets, namely:
\begin{align}
    &E_{w^{(d)}_{i}} = A\sqcup B \sqcup C \\
    & A\colonequals \{(j \ j+1) \ \mid \ 1\leq j \leq n-d-1 \} \\
& B\colonequals \{(n-d-1 \ \ n+1-j) \ \mid \ 1\leq j\leq i-1  \} \\
& C\colonequals \{(n-d \ \ n-d+j) \ \mid \ 1\leq j\leq d-i+1  \}.
\end{align}
For $d=n-1$, the number of lower neighbours of $w^{(n-1)}_{i}$ is $n-i$, and we have
\begin{equation}
    E_{w^{(d)}_{i}} =  \{(1 \ j) \ \mid \ 2\leq j \leq n-i+1 \}.
\end{equation}

Next, we compute the weight of a Hasse invariant which cuts out the stratum $\overline{Y}_{w^{(d)}_{i+1}}$ in the stratum $\overline{Y}_{w^{(d)}_{i}}$ for $1\leq i \leq d$ and $1\leq d<n-1$. By construction, we have $w_{i+1}^{(d)}=w^{(d)}_{i}s_{\alpha_{i}^{(d)}}$ for the root $\alpha_{i}^{(d)}\colonequals e_{n-d}-e_{n+1-i}$. Recall that for any $w\in W$, the Hasse section $\Ha_{w,\chi}$ is a section of $\Vcal_{\flag}(h_w(\chi))$ whose divisor has multiplicity $\langle\chi,\alpha^\vee\rangle$ along $\Fcal_{ws_\alpha}$ for each $\alpha\in E_w$ (see section \ref{subsec-Hasse-sections}). We call $h_w(\chi)$ the weight of $\Ha_{w,\chi}$. Consider the character 
\begin{equation}
\chi^{(d)}_{i}\colonequals -e_{d-i+1}. 
\end{equation}
It satisfies 
\begin{equation}
\begin{cases}
\langle \chi^{(d)}_i,\alpha^\vee\rangle =1 & \textrm{for } \alpha=\alpha^{(d)}_i  \\
\langle \chi^{(d)}_i,\alpha^\vee\rangle =0 & \textrm{for } \alpha\in E_{w^{(d)}_i}\setminus \{ \alpha^{(d)}_i\}.
\end{cases}
\end{equation}
Therefore, the partial Hasse invariant $\Ha^{(d)}_i$ on $\overline{\Fcal}_{w^{(d)}_{i}}$ cuts out with multiplicity one the the stratum $\overline{\Fcal}_{w^{(d)}_{i+1}}$. Similarly, the pullback to $Y$ is a section over $\overline{Y}_{w^{(d)}_{i}}$
which cuts out the stratum $\overline{Y}_{w^{(d)}_{i+1}}$. We denote the weight of $\Ha^{(d)}_i$ by $\ha^{(d)}_i\colonequals h_{w^{(d)}_{i}}(\chi^{(d)}_{i})$. We obtain:
\begin{equation}
    \ha^{(d)}_i = e_{d-i+1} - q w_{0,I}(e_i).
\end{equation}

\begin{proposition}\label{prop-ha-Cmin} Define $\lambda_{a,b}\in \ZZ^n$ by $\lambda_{a,b}\colonequals e_a-qe_b$ where $1\leq a,b\leq n$. Then $\lambda_{a,b}\in C_{L-\Min}$ if and only if $b\leq r$.
\end{proposition}

\begin{proof}
Assume $a\leq r$. Let $x=(x_j)_{1\leq j \leq s}$ be a finite sequence such that $r\geq x_1\geq x_2\geq\dots \geq x_s \geq 0$. We need to show that $\Gamma_x(\lambda_{a,b})\leq 0$. Write $\lambda_{a,b}=(y_1,\dots,y_n)$. We have:
\begin{equation}
    \Gamma_x(\lambda_{a,b})=\sum_{j=r+1}^n \left( \sum_{i=1}^{r-x_{j-r}} (y_i-y_j) + \frac{1}{q}\sum_{i=r-x_{j-r}+1}^r (y_i-y_j)\right).
\end{equation}
Since $b\leq r$, the sum $\sum_{i=1}^{d} (y_i-y_j) + \frac{1}{q}\sum_{i=d+1}^r (y_i-y_j)$ is $\leq 0$ for any $1\leq d\leq r$. This shows that $\lambda_{a,b}\in C_{L-\Min}$. We leave the converse implication to the reader, as we will not use it.
\end{proof}

\begin{corollary}\label{cor-weight-ha-min}
For any $d\leq \min(r,n-1)$ and any $1\leq i \leq d$, one has $\ha^{(d)}_i\in C_{L-\Min}$.
\end{corollary}

\begin{proof}
We have $\ha^{(d)}_i=e_{d-i+1} - q w_{0,I}(e_i)$. Since $i\leq d\leq r$, we have $w_{0,I}(e_i)\leq r$. The result follows from Proposition \ref{prop-ha-Cmin}.
\end{proof}

Hence, when $(r,s)=(n-1,1)$, we obtain a path from $\Lambda_1=w_0$ to $\Lambda_{n-1}$ such that each element of the sequence admits a system of partial Hasse invariants, and furthermore the weights $\ha_i^{(d)}$ (for all $1\leq i\leq d\leq n-1$) all lie in $C_{L-\Min}$.

\subsection{Hasse-regularity}

In the case $(r,s)=(n-1,1)$, we have $\Lambda_{n-1}=z$. Recall that for a general cocharacter datum $(G,\mu)$ over $\FF_q$, the element $z$ is defined by $z\colonequals \sigma(w_{0,I})w_0$ (see section \ref{subsec-not}). The last ingredient of our proof will be to show that the stratum $Y_z$ is Hasse-regular (Definition \ref{def-regular}). Before we show this, we collect in this section some expectations in the general case. 

Let $(G,\mu)$ be a general cocharacter datum $(G,\mu)$ over $\FF_q$ and $(X,\zeta)$ satisfying Assumption \ref{assume-atp}. In the terminology of \cite[Definition 2.4.2]{Goldring-Koskivirta-Strata-Hasse}, $z\colonequals \sigma(w_{0,I})w_0$ is the cominimal element of maximal length. We recall some results from \loccit about the stratum $\Fcal_z$. First, by \cite[Proposition 2.2.1]{Koskivirta-Normalization} the projection map $\pi\colon \GF^{\mu}\to \GZip^\mu$ restricts to a finite etale map $\Fcal_z\to \Ucal_\mu$, where $\Ucal_\mu$ is the open stratum of $\GZip^\mu$. On the Zariski closure, the map $\pi\colon \overline{\Fcal}_z\to \GZip^\mu$ is not finite in general. Similar results hold for the stratum $Y_z\subset Y$ and the projection map $\pi_Y\colon Y_z\to X$. We conjecture the following in general:

\begin{conjecture}\label{conj-Hasse-reg-z}
The flag stratum $Y_z$ is Hasse-regular.
\end{conjecture}

For example, take $G=\Res_{\FF_{q^m}/\FF_q}(\GL_{2,\FF_{q^m}})$ endowed with the parabolic $P=B$. This corresponds to the case of Hilbert--Blumenthal Shimura varieties. In this case, the flag space $Y=\Flag(X)$ coincides with $X$. Hence $Y_z$ is simply the unique open stratum of $X$, and we have $\overline{Y}_z=X$. In particular, Conjecture \ref{conj-Hasse-reg-z} says in this case that $\langle C_{X} \rangle = C_{\Hasse}$, which was indeed proved in \cite{Goldring-Koskivirta-global-sections-compositio}. In the case when $G$ is $\FF_q$-split, the Hasse cone of $z$ has a simple form:
\begin{equation}
    C_{\Hasse,z}=\{ \lambda\in X^*(T) \ \mid \ \langle \lambda,\alpha^\vee\rangle\leq 0 \ \textrm{for all} \ \alpha\in \Phi^+\setminus \Phi_{L,+} \}.
\end{equation}
Furthermore, in this case we expect the following stronger version:
\begin{conjecture}\label{conj-Hasse-reg-w}
Assume that $G$ is $\FF_q$-split. For any $w\in W$ such that $w\leq z$, the flag stratum $Y_w$ is Hasse-regular.
\end{conjecture}
Conjecture \ref{conj-Hasse-reg-w} holds for Hilbert--Blumenthal Shimura varieties at a split prime $p$ by \cite{Goldring-Koskivirta-global-sections-compositio}. Furthermore, it also holds for the groups $G=\Sp(4)_{\FF_q}$ and $G=\GL_{3,\FF_q}$ (in signature $(2,1)$) by \loccit (\S5.2, Figure 1 and Figure 2). For $G=\GL_{4,\FF_q}$ with a parabolic of type $(3,1)$, it follows from \cite[\S5.2]{Goldring-Koskivirta-divisibility}. We will generalize the result to the case $G=\GL_{n,\FF_q}$ with a parabolic of type $(n-1,1)$ in the next section.

\subsection{The unitary case of signature $(n-1,1)$ at split primes}\label{sec-n1}

We now return to the case $G=\GL_{n,\FF_q}$ and we consider the case $(r,s)=(n-1,1)$. In this case, the element $z$ coincides with $\Lambda_{n-1}$. We say that a permutation $w\in S_n$ is $z$-small if $w\leq z$. Similarly, a stratum $Y_w$ paramatrized by such an element will be called $z$-small. 

\subsubsection{Hasse cones of $z$-small strata}
For an integer $m\geq 1$, we consider the $m\times m$-matrix
\begin{equation}
   \left( \begin{matrix}
        & 1& &   \\
        & & \ddots&  \\
        & & &1 \\
        1& & & 
    \end{matrix}\right)
\end{equation}
which we simply denote by $[m]$ (when no confusion arises from this notation). Similarly, for a tuple of positive integers $(m_1,\dots,m_k)$, we define
\begin{equation}
    [m_1,\dots,m_k]\colonequals \left(
    \begin{matrix}
    [m_1] & &\\ &\ddots & \\ && [m_k]
    \end{matrix}
    \right).
\end{equation}
By Proposition \ref{prop-low-neigh}, the $z$-small elements of $S_n$ are precisely the permutations of the form $[m_1,\dots,m_k]$ for positive integers $m_1,\dots,m_k$ such that $m_1+\dots+m_k=n$. Note that any lower neighbour of a $z$-small element is again $z$-small. It is clear that a $z$-small element admits a system of partial Hasse invariants, because each block $[m_i]$ admits such a system.

We compute the Hasse cone $C_{\Hasse,w}$ for each $z$-small element $w$. For $w=[m_1,\dots,m_k]$, we put $M_i(w)\colonequals \sum_{d=1}^i m_i$ for $1\leq d\leq k$ and $M_0(w)\colonequals 0$. If the choice of $w$ is clear, we simply write $M_i$ instead of $M_i(w)$. The set $E_w$ is given by
\begin{equation}
    E_w=\bigsqcup_{i=1}^{k}E_{w}^{(i)}, \quad E_w^{(i)}\colonequals \{ (M_{i-1}+1 \ \ M_{i-1}+j) \mid \ 1<j\leq m_{i} \}.
\end{equation}
We say that $w'$ is an $i$-lower neighbour if it corresponds to an element of $E_w^{(i)}$, i.e if $w'=w s_\alpha$ for $\alpha\in E_{w}^{(i)}$. In other words, an $i$-lower neighbour of $w$ amounts to a partition $m_i=a+b$ with $a,b\geq 1$. For $w\in S_n$ $z$-small, put $\gamma_w\colonequals w^{-1}z$. If $w=[m_1,\dots,m_k]$, we have:
\begin{equation}
    \gamma_w=(1 \ \ M_{k-1}+1 \ \ M_{k-2}+1 \ \ \dots \ \ M_1+1).
\end{equation}
In particular, $\gamma_w$ is a $k$-cycle, so it has order $k$ in $S_n$. The cone $\langle C_{\Hasse,w} \rangle$ is defined by a number of $|E_w|$ inequalities. The inequality corresponding to $\alpha\in E_w$ is 
\[\sum_{d=0}^{k-1} q^{k-1-d} \langle z^{-1}\lambda, \gamma_w^{d} \alpha^{\vee} \rangle \geq 0\]
where $\lambda\in \ZZ^n$. For $\alpha\in E_w$, write $C_{\Hasse,w}^{\alpha}$ for the cone in $\ZZ^n$ defined by this condition. Therefore, $\langle C_{\Hasse,w} \rangle =\bigcap_{\alpha\in E_w}C_{\Hasse,w}^{\alpha}$. To simplify, we always write $z^{-1}\lambda=(x_1,\dots , x_n)\in \ZZ^n$. Let $f$ be a linear polynomial in the variables $x_1,\dots, \widehat{x_{i}}, \dots x_n$ (where $\widehat{x_i}$ means that we omit the variable $x_i$). We write $f(x_1,\dots, \widehat{x_{i}}, \dots x_n)\leq_i 0$ for the homogeneous inequality $f(x_1-x_i, \dots , x_{n}-x_i)\leq 0$. If $w=[m_1,\dots,m_k]$ and $\alpha=(M_{i-1}+1 \ \ M_{i-1}+j)$ for $1<j\leq m_{i}$, the corresponding inequality defining $C_{\Hasse,w}^{\alpha}$ is given by
\begin{equation}
    \sum_{d=1}^{i-1} q^{k-d} x_{M_{i-d}+1} + q^{k-i} x_1 + \sum_{d=i}^{k-1} q^{d-i} x_{M_{d}+1} \quad \leq_{M_{i-1}+j} 0.
\end{equation}

\subsubsection{Intersection cones}
The goal of this section is to show the following result:

\begin{proposition}\label{prop-inter-cone-Un1}
Let $w\in S_n$ be a $z$-small permutation of length $\ell(w)\geq 2$ and let $\alpha\in E_w$. There exist two lower neighbours $w_1,w_2$ of $w$ (depending on $\alpha$) such that
\begin{equation}
    C_{\Hasse,w_1}\cap C_{\Hasse,w_2}\subset C^{\alpha}_{\Hasse,w}.
\end{equation}
\end{proposition}
We write $w=[m_1,\dots,m_k]$ and $\alpha=(M_i+1 \ \ M_{i}+j)$ for $0\leq i <k$ and $1<j\leq m_{i+1}$. There are several cases to consider.

\paragraph{The case $j\geq 3$.}
In this case, we show that we may take $w_1$ and $w_2$ to be $i$-lower neighbours of $w$. Put:
\begin{align}
& w_1\colonequals [m_1,\dots,m_{i-1}, 1, m_{i}-1,m_{i+1},\dots ,m_k] \\
& w_2 \colonequals [m_1,\dots,m_{i-1}, j-1,m_{i}-j+1,m_{i+1},\dots ,m_k] 
\end{align}
In other words, $w_1$, $w_2$ are given respectively by partitioning $m_i$ into $[1,m_i-1]$ and $[j-1,m_{i}-j+1]$. Note that by assumption $j-1\geq 2$. Consider the roots:
\begin{align}
  &  \alpha_1\colonequals (M_{i-1}+2 \ \ M_{i-1}+j) \\
  &  \alpha_2\colonequals (M_{i-1}+1 \ \ M_{i-1}+2). 
\end{align}
It suffices to show $C^{\alpha_1}_{\Hasse,w_1}\cap C^{\alpha_2}_{\Hasse,w_2}\subset C^{\alpha}_{\Hasse,w}$. The equations satisfied by $C^{\alpha_1}_{\Hasse,w_1}$ and $C^{\alpha_2}_{\Hasse,w_2}$ are respectively:
\begin{align}
 &(E_1):\quad  q^k x_{M_{i-1}+2} + \sum_{d=1}^{i-1} q^{k-d} x_{M_{i-d}+1} + q^{k-i} x_1 + \sum_{d=i}^{k-1} q^{d-i} x_{M_{d}+1}  &\leq_{M_{i-1}+j} 0 \\
 &(E_2):\quad  \sum_{d=0}^{i-2} q^{k-d} x_{M_{i-1-d}+1} + q^{k-i+1} x_1 + \sum_{d=i}^{k-1} q^{d-i+1} x_{M_{d}+1} + x_{M_{i-1}+j}  &\leq_{M_{i-1}+2} 0.
\end{align}
Equation $(E_1)$ is very similar to the one defining $C^{\alpha}_{\Hasse,w}$, except for the presence of the leading term $q^k x_{M_{i-1}+2}$. We can remove this term by using a linear combination with the second inequality (recall that the variable $x_{M_{i-1}+2}$ appears in $(E_2)$ by definition of the symbol $\leq_{M_{i-1}+2}$). Specifically, put $\delta\colonequals \frac{q^k}{\sum_{j=0}^k q^j}=\frac{q^k(q-1)}{q^{k+1}-1}$. Since $\delta$ is positive, we may form the inequality $(E_1)+\delta (E_2)$. Dividing throughout by $1+\delta q$, we obtain precisely the inequality for $C^{\alpha}_{\Hasse,w}$.

\paragraph{The case $j=2$ and $m_i>2$.}
In this case too, we may take $w_1$ and $w_2$ to be $i$-lower neighbours of $w$. Put:
\begin{align}
w_1\colonequals [m_1,\dots,m_{i-1}, 2, m_{i}-2,m_{i+1},\dots ,m_k] \\
w_2 \colonequals [m_1,\dots,m_{i-1}, 1, m_{i}-1,m_{i+1},\dots ,m_k] 
\end{align}
In other words, $w_1$, $w_2$ are given respectively by partitioning $m_i$ into $[2,m_i-2]$ and $[1,m_{i}-1]$. Consider the roots:
\begin{align}
   & \alpha_1\colonequals \alpha = (M_{i-1}+1 \ \ M_{i-1}+2) \\
    & \alpha_2\colonequals (M_{i-1}+2 \ \ M_{i-1}+3). 
\end{align}
It suffices to show $C^{\alpha_1}_{\Hasse,w_1}\cap C^{\alpha_2}_{\Hasse,w_2}\subset C^{\alpha}_{\Hasse,w}$. The equations satisfied by $C^{\alpha_1}_{\Hasse,w_1}$ and $C^{\alpha_2}_{\Hasse,w_2}$ are respectively:
\begin{align}
 &(E_1):\quad  \sum_{d=0}^{i-2} q^{k-d} x_{M_{i-d-1}+1} + q^{k-i+1} x_1 + \sum_{d=i}^{k-1} q^{d-i} x_{M_{d}+1} + x_{M_{i-1}+3}  &\leq_{M_{i-1}+2} 0 \\
 &(E_2):\quad q^k x_{M_{i-1}+2} + \sum_{d=1}^{i-1} q^{k-d} x_{M_{i-d}+1} + q^{k-i} x_1 + \sum_{d=i}^{k-1} q^{d-i} x_{M_{d}+1}  &\leq_{M_{i-1}+3} 0.
\end{align}
Equation $(E_1)$ is very similar to the one defining $C^{\alpha}_{\Hasse,w}$ (multiplied by $q$), except for the presence of the last term $ x_{M_{i-1}+3}$ in $(E_1)$. We can remove this term by using a linear combination with the second equation. Specifically, put $\delta\colonequals \frac{1}{\sum_{j=0}^k q^j}=\frac{(q-1)}{q^{k+1}-1}$. Since $\delta$ is positive, we have the inequality $(E_1)+\delta (E_2)$. Dividing throughout by $1+\delta q$, we obtain precisely the inequality for $C^{\alpha}_{\Hasse,w}$.

\paragraph{The case $j=2$ and $m_i=2$.}
In this case, $w$ admits only one $i$-lower neighbour, namely
\begin{equation}
w_2\colonequals [m_1,\dots,m_{i-1}, 1, 1,m_{i+1},\dots ,m_k] 
\end{equation}
(which corresponds to the partition $2=1+1$). Therefore, we need to choose $w_1$ in a different block. Since we assume $\ell(w)\geq 2$, at least one other $m_j$ is $\geq 2$. We take
\begin{equation}
w_1\colonequals [m_1,\dots,m_{j-1}, 1, m_j-1 ,m_{j+1},\dots ,m_k] 
\end{equation}
(the $j$-lower neighbour corresponding to the partition of $m_j$ into $[1,m_j-1]$). Set:
\begin{align}
    &\alpha_1\colonequals \alpha = (M_{i-1}+1 \ \ M_{i-1}+2) \\
&\alpha_2\colonequals  (M_{j-1}+1 \ \ M_{j-1}+2).
\end{align}
It suffices to show $C^{\alpha_1}_{\Hasse,w_1}\cap C^{\alpha_2}_{\Hasse,w_2}\subset C^{\alpha}_{\Hasse,w}$. Assume first that we can choose $j>i$. The equations satisfied by $C^{\alpha_1}_{\Hasse,w_1}$ and $C^{\alpha_2}_{\Hasse,w_2}$ are respectively:
\begin{align}
 (E_1):\quad  \sum_{d=0}^{i-2} q^{k-d} x_{M_{i-d-1}+1} + q^{k-i+1} x_1 + \sum_{d=j}^{k-1} q^{d-i+1} x_{M_{d}+1} & + q^{j-i} x_{M_{j-1}+2} \\
&+ \sum_{d=i}^{j-1} q^{d-i} x_{M_{d}+1} \quad  \leq_{M_{i-1}+2} 0.
\end{align}
\vspace{-1cm}
\begin{align}
 (E_2):\quad  \sum_{d=0}^{j-i-1} q^{k-d} x_{M_{d+i}+1} + q^{k-j+i} x_{M_{i-1}+2} + \sum_{d=1}^{i-1} q^{k-j+d} x_{M_{d}+1} & + q^{k-j} x_{1} \\
&+ \sum_{d=j}^{k-1} q^{d-j} x_{M_{d}+1} \ \leq_{M_{j-1}+2} 0.
\end{align}
Equation $(E_1)$ is similar to the one defining $C^{\alpha}_{\Hasse,w}$. Specifically, the last terms $x_{M_{d}+1}$ for $i\leq d\leq j-1$ are the same in both equations. The terms $x_{M_{d}+1}$ for all other $d$ and for $x_1$ are multiplied by an extra power of $q$ in equation $(E_1)$. Finally, the term $q^{j-i}x_{M_{j-1}+2}$ in $(E_1)$ does not appear in the equation of $C^{\alpha}_{\Hasse,w}$. Using a similar strategy as before, we remove this term by using a linear combination with the second equation $(E_2)$. Put $\delta\colonequals \frac{q^{j-i}}{\sum_{d=0}^k q^{d}}=\frac{q^{j-i}(q-1)}{q^{k+1}-1}$. Since $\delta$ is positive, we have the inequality $(E_1)+\delta (E_2)$. In this equation, the variable $x_{j-1}+2$ has disappeared. We write the terms in decreasing order of the power of $q$ as they appear in the equation of $C^{\alpha}_{\Hasse,w}$, namely $x_{M_{i-1}+1}, x_{M_{i-2}+1}, \dots, x_1, x_{M_{k-1}}, \dots, x_{M_j+1}, x_{M_{j-1}+1}, \dots, x_{M_i+1}$. One sees immediately that the coefficents which appear in front of these terms in $(E_1)+\delta (E_2)$ are divided by $q$ at each step between the terms $x_{M_{i-1}+1}$ and $x_{M_j+1}$, and between
$x_{M_{j-1}+1}$ and $x_{M_i+1}$. It remains to show that the same happens between the terms $x_{M_j+1}$ and $x_{M_{j-1}+1}$. The coefficient of $x_{M_j+1}$ is $q^{j-i+1}+\delta$, and the coefficient of $x_{M_{j-1}+1}$ is $q^{j-i-1}+\delta q^k$. Since $\delta=\frac{q^{j-i}(q-1)}{q^{k+1}-1}$, one has indeed $q^{j-i+1}+\delta=q(q^{j-i-1}+\delta q^k)$. This shows that the equation $(E_1)+\delta (E_2)$ is a positive multiple of the equation for $C^{\alpha}_{\Hasse,w}$. 

It remains to treat the case when there is no $j>i$ such that $m_j\geq 2$. We choose $j<i$ with $m_j\geq 2$, and define $w_1$, $w_2$, $\alpha_1$, $\alpha_2$ as before. The equations satisfied by $C^{\alpha_1}_{\Hasse,w_1}$ and $C^{\alpha_2}_{\Hasse,w_2}$ are respectively:
\begin{align}
 (E_1):\quad  \sum_{d=j}^{i-1} q^{k-i+d+1} x_{M_{d}+1} + q^{k-i+j} x_{M_{j-1}+2} + \sum_{d=1}^{j-1} q^{k-i+d} x_{M_{d}+1} & + q^{k-i} x_{1} \\
&+ \sum_{d=i}^{k-1} q^{d-i} x_{M_{d}+1} \ \leq_{M_{i-1}+2} 0.
\end{align}
\vspace{-1cm}
\begin{align}
 (E_2):\quad  \sum_{d=1}^{j-1} q^{k-j+d+1} x_{M_{d}+1} + q^{k-j+1} x_{1} + \sum_{d=i}^{k-1} q^{d-j+1} x_{M_{d}+1} & + q^{i-j} x_{M_{i-1}+2} \\
&+ \sum_{d=j}^{i-1} q^{d-j} x_{M_{d}+1} \ \leq_{M_{j-1}+2} 0.
\end{align}
As before, we remove the term $x_{M_{j-1}+2}$ in $(E_1)$ using $(E_2)$. Put $\delta\colonequals \frac{q^{k-i+j}}{\sum_{d=0}^k q^d}=\frac{q^{k-i+j}(q-1)}{q^{k+1}-1}$ and consider $(E_1)+\delta (E_2)$. Again, the coefficients of $x_{M_{i-1}+1}, x_{M_{i-2}+1}$, $\dots, x_1, x_{M_{k-1}}, \dots$, $x_{M_j+1}, x_{M_{j-1}+1}$, $\dots, x_{M_i+1}$ (in this order) are divided by $q$ at each step, except perhaps for the coefficients of $x_{M_j+1}$ and $x_{M_{j-1}+1}$. The former is $q^{k-i+1+j}+\delta$ and the latter is $q^{k-i+j-1}+q^k\delta$. Again, we have $q^{k-i+1+j}+\delta = q (q^{k-i+j-1}+q^k\delta)$ by definition of $\delta$. This shows the result.

\subsubsection{Main result}\label{main-sec-n1}
Our first main result is the strong version of the Hasse-regularity conjecture (see Conjecture \ref{conj-Hasse-reg-w}) for unitary Shimura varieties of good reduction at a split prime. More generally, we take $(X,\zeta)$ to be an arbitrary pair satisfying Assumption \ref{assume-atp}.

\begin{theorem}\label{Un1-main-thm}
Assume $G=\GL_{n,\FF_q}$ and $(r,s)=(n-1,1)$. For any $z$-small element $w\in S_n$, the flag stratum $Y_w$ is Hasse-regular.
\end{theorem}

\begin{proof}
Since all $z$-small strata admit a system of Hasse invariants (Definition \ref{syst-Hasse}), we may construct a separating system $\EE=(\EE_w)_{w\in W}$ as follows. For $z$-small elements $w\in W$, we set $\EE_w=E_w$ and we let $\{\chi\}_{\alpha\in E_w}$ be any system of characters satisfying Conditions (a) and (b) of Definition \ref{sep-syst-E}. For $w$ not $z$-small, we set $\EE_w=\emptyset$. We show by induction on $\ell(w)$ that for all $z$-small element $w$, the intersection cone $C^{+,\EE}_w$ satisfies $C^{+,\EE}_w \subset \langle C_{\Hasse,w} \rangle$. For $\ell(w)=1$ the result holds by Lemma \ref{lw1-reg}. 
Suppose the result holds for all $z$-small strata of length $\leq d$ and let $w$ be a $z$-small element of length $\ell(w)=d+1$. By Proposition \ref{prop-inter-cone-Un1}, we obtain
\begin{equation}
\bigcap_{\alpha\in E_w} C^{+,\EE}_{ws_{\alpha}} \subset \bigcap_{\alpha\in E_w} \langle C_{\Hasse, ws_{\alpha}}\rangle \subset \langle C_{\Hasse,w} \rangle
\end{equation}
Since we clearly have $C^{\EE}_{\Hasse,w}\subset C_{\Hasse,w}$, we deduce $C^{+,\EE}_w\subset \langle C_{\Hasse,w} \rangle$, which proves the result. By Theorem \ref{thm-sep-syst}, we deduce that for any $z$-small element,
$\langle C_{Y,w} \rangle = \langle C^{+,\EE}_w \rangle = \langle C_{\Hasse,w} \rangle$. This terminates the proof.
\end{proof}

In particular, for the element $w=z$, we deduce the following:
\begin{corollary}
We have $\langle C_{Y,z} \rangle =\{ (k_1,\dots,k_n)\in \ZZ^n \ \mid \ k_i-k_n\leq 0 \ \textrm{ for all } \ i=1,\dots , n \}$.
\end{corollary}\label{Un1-cor}
We also deduce from Theorem \ref{Un1-main-thm} the following approximation of the cone $\langle C_K(\overline{\FF}_p) \rangle$:

\begin{theorem}\label{th-CF-Min}
We have $C_K(\overline{\FF}_p) \subset C_{L-\Min}$. In other words, the weight $(k_1,\dots,k_n)$ of any nonzero mod $p$ automorphic form satisfies: 
\begin{equation}
    \sum_{i=1}^j (k_i-k_n) + \frac{1}{p} \sum_{i=j+1}^{n-1} (k_i-k_n) \leq 0 \quad \textrm{ for all }j=1,\dots, n-1.
\end{equation}
\end{theorem}
\begin{proof}
We consider the sequence $(w_i^{(d)})_{i,d}$ for $1\leq i \leq d+1$ and $1\leq d < n-1$ which defines a path (in the terminology of section \ref{subsec-auxil}) from $\Lambda_1=w_0$ to $\Lambda_{n-1}=z$. By Corollary \ref{Un1-cor}, we have $C_{Y,z}\subset C_{L-\Min}$. Furthermore, by Corollary \ref{cor-weight-ha-min}, the weight of the partial Hasse invariant $\Ha_i^{(d)}$ which cuts out the stratum $Y_{w_{i+1}^{(d)}}$ (for $1\leq i \leq d$) in the closure of $Y_{w_i^{(d)}}$ lies in $C_{L-\Min}$. We deduce that $C_{Y,w}\subset C_{L-\Min}$ for each $w$ in the chain. In particular, the result holds for $w_0$, which terminates the proof.
\end{proof}
Theorem \ref{th-CF-Min} illustrates again the connection between group theory and geometry of Shimura varieties: The cone $C_{L-\Min}$ originates from a unipotent-invariance condition for automorphic forms on $\GZip^\mu$. Theorem \ref{th-CF-Min} and its proof show that this condition also appears geometrically as a relationship between the flag strata of a Shimura variety.

Finally, we note that Theorem \ref{th-CF-Min} provides a second, more precise proof of the containment $C_K(\CC)\subset C_{\GS}$. Indeed, let $\Sscr_K$ be an integral Shimura variety of Hodge-type of unitary type and signature $(n-1,1)$. At each split prime $p$ of good reduction, we have $C_K(\overline{\FF}_p) \subset C_{L-\Min,p}$, where $C_{L-\Min,p}$ denotes the $L$-minimal cone of the induced zip datum at $p$. We obtain 
\[C_K(\CC)\subset \bigcap_{\textrm{split }p} C_{K}(\overline{\FF}_p) \subset \bigcap_{\textrm{split }p} C^{+,I}_{L-\Min,p}=C_{\GS}.\]

\bibliographystyle{test}
\bibliography{biblio_overleaf}

\noindent
Wushi Goldring\\
Department of Mathematics, Stockholm University, SE-10691, Sweden\\
wushijig@gmail.com\\ 

\noindent
Jean-Stefan Koskivirta\\
Department of Mathematics, Faculty of Science, Saitama University, 
255 Shimo-Okubo, Sakura-ku, Saitama City, Saitama 338-8570, Japan \\
jeanstefan.koskivirta@gmail.com

\end{document}